\documentclass[12pt,a4paper,reqno]{amsart}
\usepackage{color,latexsym,amsfonts,amssymb,bbm,comment,amsmath,cite, amsthm, bbold, float, enumitem}

\usepackage{hyperref}
\usepackage{fullpage}
\usepackage{graphicx}
\usepackage{tikz}
\usepackage{dsfont}
\usepackage{chngcntr}

\numberwithin{equation}{section}
\numberwithin{figure}{section}

\newtheorem{theorem}{Theorem}
\newtheorem{corollary}{Corollary}[section]

\newtheorem{lemma}{Lemma}[section]
\newtheorem{proposition}{Proposition}[section]
\theoremstyle{definition}
\newtheorem{definition}{Definition}[section]

\usepackage{dsfont, color}
\DeclareFontFamily{T1}{dutchcal}{}
\DeclareFontShape{T1}{dutchcal}{m}{n}{<->s*[1.44]callig15}{}
\DeclareMathAlphabet\mathdutchcal   {T1}{dutchcal} {m} {n}

\newcommand{\Lbrace}{\left \lbrace}
\newcommand{\Rbrace}{\right \rbrace}

\newcommand*{\E}{\mathbb{E}}

\newcommand*{\N}{\mathbb{N}}
\renewcommand{\P}{\mathbb{P}}

\newcommand*{\R}{\mathbb{R}}

\newcommand*{\Z}{\mathbb{Z}}

\newcommand{\comp}[2]{\mathcal{C}_{#2}^{#1}}
\newcommand{\explor}[2]{\mathcal{Z}_{#2}^{#1}}
\newcommand{\cal}[1]{\mathcal{#1}}
\newcommand{\escball}[1]{T_{B_{#1}}}

\newcommand{\BR}[2]{(B)_{#1}^{#2}}
\newcommand{\hittime}[1]{H_{#1}}
\newcommand{\exittime}[1]{T_{#1}}
\newcommand{\conditionH}{(\mathcal{H})}

\newcommand{\Pq}[1]{P_{#1, \omega}}


\begin{document}

\author{Alejandro F. Ram\'\i rez$^{1,2}$}
\author{Rodrigo Ribeiro$^{1,3}$}

\title{Computable criteria for ballisticity of random walks in elliptic random environment}

\date{\today \\
$^1$ Universidad Cat\'{o}lica de Chile \\
$^2$ NYU-Shanghai \\
$^3$ Colorado University Boulder
}

\begin{abstract}  We consider random walks in i.i.d. elliptic random
  environments which
  are not uniformly elliptic. We introduce a computable condition
  in dimension $d=2$ and a general condition valid for
  dimensions $d\ge 2$ expressed in terms of the exit time from
  a box,
  which ensure that local trapping would not inhibit a ballistic
  behavior
  of the random walk. An important technical innovation related to our
  computable condition, is the introduction of a geometrical point of
  view to classify the way in which the random walk can become
  trapped,
  either in an edge, a wedge or a square.
  Furthermore, we prove that the general condition we introduce is sharp.
\end{abstract}

\maketitle

\section{Introduction}
Finding explicit computable criteria giving information about the long-time behavior of 
a random walk in random i.i.d. environment on $\mathbb Z^d$ for $d\ge 2$ is a challenging 
problem  for which few partial results have been derived. A particular  question 
of this kind, is under what criteria can we ensure that the behavior of the random 
walk is ballistic (i.e. with non-zero velocity). It is natural to expect that for dimensions $d\ge 2$, a criteria which would imply ballisticity should be 
 directional transience. In the uniform elliptic case, 
a family of conditions which correspond to a priori strong forms of directional transience and which do imply ballisticity, 
were introduced in a series of works including \cite{Sz01,Sz02} and \cite{BDR14}. 
These conditions are defined in terms of the velocity of decay of the exit 
probability through the atypical side of the slab: 
condition $(T)$ (exponential decay) and $(T')$ (almost exponential decay), 
both introduced 
by Sznitman in \cite{Sz01,Sz02}, and the polynomial condition $(P)_M$ introduced by Berger, Drewitz 
and Ram\'\i rez in \cite{BDR14}. All of them have been proven 
to be equivalent  (see \cite{BDR14} and \cite{GR20}). Condition 
$(P)_M$ can be verified for some environments (see for example \cite{Sz03,RS19,FR20}), and it is 
of global nature, in the sense that it should be verified in finite, but large boxes. As soon 
as the uniform ellipticity condition is relaxed to just ellipticity, a new kind of phenomena 
appears, where local traps corresponding to just a few edges could inhibit 
ballisticity even if the random walk is directional transient. A local trap here 
means that the walk in average takes an infinite time to escape a 
region including just a few sites. In this case, the ballisticity 
criteria $(T)$, should be complemented with an additional condition which 
inhibits the appearance of local traps produced by strong tails near 
degeneracy of the elliptic environment. 

The first result of this article is the introduction of a computable ellipticity condition 
in dimension $d=2$ (condition $(X)$), in the 
sense that that it can be explicitly checked just 
knowing the law of the environment at a single site.
This condition is expressed in a geometrical way, in terms of the exit time from an edge, and a set of extra conditions related to the exit time from a wedge and from a square, together with a requirement involving correlations between the jump probabilities at a single site.
We show that 
condition $(X)$ together with condition $(P)_M$ imply ballisticity.

A second result presented here is a general second condition (condition $(B)$), valid for $d\ge 2$, 
and expressed in terms of the expected exit time from a large box, 
which also together with $(P)_M$ implies ballisiticity. We also show
that condition $(B)$ is sharp.

\medskip
\subsection{Ballisticity for random walks in elliptic random environments} Let us introduce the random walk in random environment model (RWRE). 
Let $U=\{e\in\mathbb Z^d:|e|_1\}$ and $\mathcal P=\{p(e): e\in U\}$
be the set of probability vectors with components in $U$. We will 
also use the notation $\{e_1,e_{-1},\ldots, e_d,e_{-d}\}$ for the 
elements of $U$, with the convention $e_{-i}=-e_i$, $1\le i\le d$. 
We define the environmental space $\Omega=\mathcal P^{\mathbb Z^d}$
and use the notation $\omega=(\omega(x))_{x\in\mathbb Z^d}\in\Omega$, 
with $\omega(x)=(\omega(x,e))_{e\in U}\in\mathcal P$. 
For $\omega$ fixed, define the Markov chain $(X_n)_{n\ge 0}$
with transition probabilities 

$$
P_\omega(X_{n+1}=y+e|X_n=y)=\omega(y,e)\quad{\rm for}\ {\rm all}\ y\in 
\mathbb Z^d, e\in U. 
$$
We will call  this Markov chain a random walk in the environment $\omega$
and denote by $P_{x,\omega}$ its law starting from $x$. Whenever $\omega$
is chosen according to some probability measure $\mathbb P$ defined on the 
environmental space $\Omega$, we call $P_{x,\omega}$ the {\it quenched}
law of the RWRE starting from $x$. Similarly we call the semidirect product $P_{x}$ defined by 
$P_x(A\times B)=\int_AP_{x,\omega}(B)d\mathbb P$, the {\it averaged}
or {\it annealed} law of the RWRE starting from $x$. 
Throughout this article we will assume that $(\omega(x))_{x\in\mathbb Z^d}$
are i.i.d. under $\mathbb P$. We say  that  $\mathbb P$ is 
{\it uniformly elliptic} if there is a constant $\kappa>0$ such that 
$\mathbb P$-a.s. we have that 

$$
\omega(x,e)\ge\kappa 
\ {\rm for}\ {\rm all}\ x\in\mathbb Z^d, e\in U, 
$$
while we say that $\mathbb P$ is 
{\it elliptic}  if $\mathbb P$-a.s. we have that 

$$
\omega(x,e)>0\ {\rm for}\ {\rm all}\ x\in\mathbb Z^d, e\in U. 
$$

One of the mayor questions about the RWRE model is the relation between 
directional transience and ballisticity. Given a direction $l\in\mathbb S^{d-1}$, 
we say that the random walk is {\it transient in direction} $l$ if a.s. we have that 

$$
\lim_{n\to\infty}X_n\cdot l=\infty. 
$$
We say that the random walk is {\it ballistic in direction} $l$
if a.s. 

$$
\liminf_{n\to\infty}\frac{X_n\cdot l}{n}>0. 
$$
It is known that for $\mathbb P$ i.i.d. and  elliptic, 
ballisiticty in a given direction implies a law of 
large numbers 

$$
\lim_{n\to\infty}\frac{X_n}{n}=v\ a.s., 
$$
with $v\ne 0$. In dimension $d=1$ it is known that for $\mathbb P$
i.i.d. and uniformly elliptic, directional 
transience does not imply ballisticity (see for example Sznitman \cite{Sz04}). 
The one-dimensional directional transient examples which are not ballistic 
are produced by laws of the environment which favor the presence of 
large (global) traps which slowdown the movement of the walk and whose 
size increases as time $n\to\infty$. On the other hand, it is expected that 
the cost of such traps in dimensions $d\ge 2$ would be too high to produce 
these examples, so that for $\mathbb P$ i.i.d. and uniformly elliptic, 
directional transience would imply ballisticity. This is still an open question. 

As a way to tackle this problem, some intermediate 
conditions which interpolate between directional transience and 
ballisticity have been introduced. For $l\in\mathbb Z^d$, 
and $\gamma\in (0,1]$, we say that condition $(T)_\gamma|l$ is 
satisfied if there is an open set $O\subset\mathbb S^{d-1}$ such that 

$$
\frac{1}{L^\gamma}\lim_{L\to\infty}P_0(X_{T_{U_{L,l'}}}\cdot l'<0)<0, 
$$
where $U_{L,l'}=\{x\in\mathbb Z^d:-L\le x\cdot l'\le x\}$ and $T_{U_{L,l'}}=
\min\{n\ge 0:X_n\notin U_{L,l'}\}$. The case $\gamma=1$ is 
called condition $(T)|l:=(T)_1|l$. While we define condition $(T')|l$
as the fulfillment of $(T)_\gamma|l$ for all $\gamma\in (0,1)$. 
These condition where introduced by Sznitman in \cite{Sz02,Sz03}. 
On the other hand, given $M\ge 1$, we say that the polynomial 
condition $(P)_M|l$ is satisfied if there is an $L_0$ such that 

$$
P_0(X_{T_{U_{L,l'}}}\cdot l'<0)\le\frac{1}{L^M}\ {\rm for}\ {\rm all}\ L\ge L_0. 
$$
This condition was defined in \cite{BDR14}. In the case of uniformly elliptic 
environments it was shown in \cite{Sz02}, \cite{BDR14} and \cite{GR20}, 
that conditions $(T)_\gamma$ for some $\gamma\in (0,1)$, $(T)$, $(T')$
and $(P)_M$ for $M\ge 15d+5$, are equivalent, and that they 
imply ballistic behavior together with an annealed and quenched central 
limit theorem. 
To extend these results to random walks in elliptic (but not necessarily 
uniformly elliptic) environments,  a minimal condition integrability condition 
has to be assumed. This can be described as a good enough 
behavior of the jump probabilities near $0$. 
We say that condition $(E)_0$
is satisfied if for all $e\in U$ there exist $\eta(e)>0$ such that 
	\begin{equation}\label{def:etastar}
	\E\left[\omega(0,e)^{-\eta(e)}\right] <\infty. 
	\end{equation}
Under $(E)_0$, the equivalence between $(T)_\gamma$
for $\gamma\in (0,1)$, 
$(T')$ and $(P_M)$ for $M\ge 15d+5$, was proven in \cite{CR14}. 
Let $\beta > 0$. 
We say that the law of the environment satisfies the ellipticity condition 
$(E)_\beta$ if there exists an $\{\alpha(e): e \in U\} \in (0,\infty)^{2d}$ such that 

$$
2\sum_{e'}\alpha(e')-\max_{e\in U}(\alpha(e)+\alpha(-e))>\beta 
$$
and for every $e\in U$

$$
\mathbb E\left[\prod_e\omega(0,e)^{-\alpha(e)}\right]<\infty. 
$$
In \cite{CR14} and \cite{BRS16} it was proved that whenever $d\ge 2$ and 
condition $(P)_M$ for $M\ge 15d+5$ together with $(E')_1$ are satisfied, then 
the random walk is ballistic.  Condition $(E')_1$ of this result is 
a sharp ellipticity condition for ballisticity for random walks in Dirichlet environments 
\cite{ST11,ST17}. It can actually be shown that whenever the law of the jump 
probabilities at a single site is asymptotically independent at small values, 
it is also a sharp condition. Nevertheless, as explained in \cite{FK16}, condition 
$(E')_1$ is not a sharp condition in general. There, the authors present 
a condition expressed in terms if the exit time from a hypercube of the random 
walk. We define a hypercube located at $x\in\mathbb Z^d$ as 

$$
H_x:=\left\{x+\sum_{i=1}^d \epsilon_i e_i,\ \epsilon_i=0\ {\rm or}\ 1\
  {\rm for}\ {\rm all}\ 1\le i\le d\right\}. 
$$
We say that condition $(C)_1$ is satisfied if 

$$
\max_{{y}\in H_x}E_{y}\left[ T_{H_x}\right]<\infty. 
$$
It was shown in \cite{FK16} that if $(C)_1$ is not satisfied the random walk 
is not ballistic. The following is an open question: 

\medskip 

\noindent Does $(C)_1$ together with $(P)_M$ for $M\ge 15d+5$ imply 
ballisiticity? 

\medskip 

An ellipticity condition which is more general that $(E')_1$, denoted 
by $(K)_1$, was defined in \cite{FK16}, where they showed that 
$(K)_1$ together with $(P)_M$ for $M\ge 15d+5$ implies ballisticity. 
In this article we will introduce a computable ellipticity condition 
in dimension $d=2$ and a simple general dimension condition, similar 
in spirit to condition $(C)_1$, 
which also imply ballistic behavior. 
\medskip

\subsection{Main results}
\label{sec:notation}
Our main result will be stated for random walks in dimensions $d=2$, providing 
a computable criteria for ballistic behavior. To state it we will 
 introduce a condition which quantifies the singularities at a site involving two or three directions 
simultaneously.  In order to preserve the visual appeal in some arguments we will make use of diagrams instead of letters $i$ and $j$ to denote those directions. For instance, diagram $\perp$ represents directions~$e_{-1}, e_1$ and $e_2$. Under such convention we define 
\begin{equation}
Q_{\llcorner}:= \max\{\omega(0, e_{1}),\omega(0, e_
2)\}, \quad Q_{\perp}:= \max\{\omega(0, e_{1}), \omega(0, e_{-1}), \omega(0, e_
2)\}
\end{equation}
and define similar quantities for all the corresponding multiple of $90$ degree rotations. We also use the following type of shorthand notation 
$$
\left\{ \max\{\vdash\right\} = \downarrow\} = \left\{\arg \max_{i \in \{-2,1,2\}}\omega(0,e_i) = e_{-2}\right\}. 
$$
Our condition requires negative moments of the above defined set of random variables and is stated as follows. We say that an i.i.d. law $\mathbb P$
  on $\Omega$ in dimension $d=2$ satisfies {\it condition $(X)_a$} if there exist $\alpha_{\lrcorner}, \beta_{\vdash}$ (and all multiples of $90$ degree rotations of $\lrcorner$ and $\vdash$), such that 
	\begin{equation}\label{x:eq1}
	\int_{\Omega} Q^{-\beta_{\vdash}}_{\vdash} {\rm d}\mathbb{P} < \infty; \quad \int_{\Omega} Q^{-\alpha_{\lrcorner}}_{\lrcorner} {\rm d}\mathbb{P} < \infty, 
	\end{equation}
	for all multiples of $90$ degree rotations of $\lrcorner$ and $\vdash$, 
	\begin{equation}\label{x:eq2}
	\int_{\Omega} Q^{-\beta_{\perp}}_{\vdash} Q^{-\alpha_{\stackrel{}{\ulcorner}}}_{\stackrel{}{\ulcorner}} \mathbb{1}_{\{ \max\{\vdash\} = \uparrow\}} {\rm d}\mathbb{P}< \infty; \; \int_{\Omega} Q^{-\beta_{\top}}_{\vdash} Q^{-\alpha_{\stackrel{}{\llcorner}}}_{\stackrel{}{\llcorner}} \mathbb{1}_{\{ \max\{\vdash\} = \downarrow\}} {\rm d}\mathbb{P}< \infty; \; \int_{\Omega} Q^{-\alpha_{\stackrel{}{\lrcorner}}}_{\stackrel{}{\lrcorner}}Q^{-\alpha_{\stackrel{}{\llcorner}}}_{\stackrel{}{\llcorner}}  {\rm d}\mathbb{P}< \infty, 
	\end{equation}
	for all multiples of $90$ degree rotations of $\stackrel{}{\ulcorner}, \vdash$ and $\uparrow$. Additionally, we  require that 
	\begin{gather}
		\beta_{\dashv} + \beta_{\vdash} > a\label{eq:edgerel} \\
		\alpha_{\stackrel{}{\ulcorner}} + \beta_{\perp} + \beta_{\dashv} > a\label{eq:wedgerel}\\
		\alpha_{\stackrel{}{\ulcorner}}	+ \alpha_{\stackrel{}{\urcorner}} + \alpha_{\lrcorner}+ \alpha_{\llcorner}> a\label{eq:squarerel}
	\end{gather}
	including all the multiple of $90$ degree rotations of
        \eqref{eq:edgerel} and \eqref{eq:wedgerel}.

We can now state the main result of this article. 
\medskip 
\begin{theorem}\label{thm:X} Consider an RWRE in $\Z^2$ whose environment satisfies conditions $(E)_0, (X)_1$ and $(P)^{\ell}_M$ for some $M>35$ and $\ell \in S^{1}$. Then, the walk is ballistic in direction $\ell$, that is $P_0$-almost surely 
	$$
	\lim_{n\to \infty}\frac{X_n}{n} = v, \; \text{ where }v\cdot\ell >0. 
	$$

      \end{theorem}

      \medskip 

 Conditions $(E)_0$ and $(X)_1$ depend only on 
      the distribution of a single site. Moreover they are computable 
      in the sense that given the distribution on a fixed site, 
      verifying such conditions is a matter of computing integrals of 
      positive functions over $\R$.
Also, condition $(X)_1$ has a geometrical interpretation.
      When $a=1$, (\ref{eq:edgerel}) 
        together with the first requirement of (\ref{x:eq1}), implies 
        that 
        the random walk cannot become trapped on any edge
        (see Figure 1.2).
		\begin{figure}[H]\label{fig:edgecase}
			\includegraphics[scale=0.4]{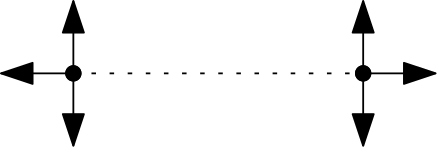}
			\caption{Schematic of an edge and the transitions pointing out it.}
		\end{figure}
		On the other hand, intuitively we would expect that 
        (\ref{eq:wedgerel}) 
        together with (\ref{x:eq1}) imply that it cannot become 
        trapped in any wedge (see Figure \ref{fig:wedgeblock}); while (\ref{eq:squarerel}) with the 
        second requirement in (\ref{x:eq1}) would imply that it cannot 
        become trapped in a square (see Figure \ref{fig:wedgeblock}).
		\begin{figure}[H]\label{fig:wedgeblock}
			\begin{center}
		\scalebox{1}{
			\begin{tabular}{cc}
				
				\includegraphics[scale=0.3]{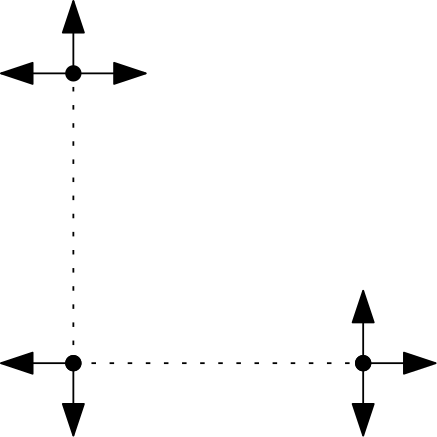}
				&
				\includegraphics[scale=0.3]{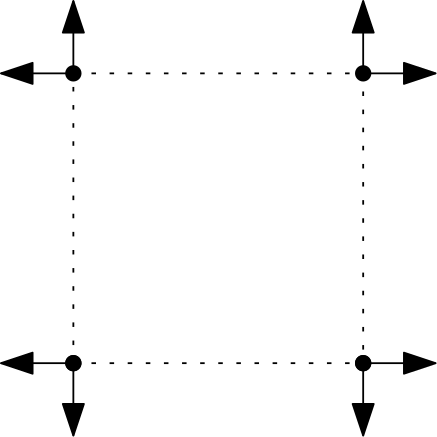}
				\\
				\multicolumn{1}{c}{ } & \multicolumn{1}{c}{} \\
				\multicolumn{1}{c}{\small Transitions point out a wedge.}
				&
				\multicolumn{1}{c}{\small Transitions pointing out a square.}
			\end{tabular}
			  }
			  \caption{\ }
		\end{center}
		\end{figure}
		Nevertheless, an important 
        observation of this article, which will be shown in Section 2, 
        is that it is possible to construct environments for which  (\ref{x:eq1}), 
        (\ref{eq:edgerel}), 
        (\ref{eq:wedgerel}) and (\ref{eq:squarerel}) are satisfied, 
        but nevertheless the random walk is not trapped in some 
        wedges or squares. This is due to the behavior of the 
        correlations 
        of the jump probabilities at a single site. For this reason, 
        we require in our definition of condition $(X)_a$ also 
        (\ref{x:eq2}).
      In Section \ref{sec:localtrapping} we will discuss
        in more  detail condition $(X)_1$ and explain the connection
        between relations \eqref{x:eq1}-\eqref{eq:squarerel}. The
        proof
        of Theorem \ref{thm:X} is based on this geometrical interpretation
        of our conditions, through the use of the theory of flow networks.

      Our second result  valid in any dimension $d\ge 2$
      ensures ballisticity under the requirement that certain moments 
      of the exit time from a box are finite.
      For any $x\in\mathbb Z^d$, we will use the standard notation
      for the norm $|x|_1=|x_1|+\cdots+|x_d|$, $|x|_2=(|x_1|^2+\cdots+|x_d|^2)^{1/2}$ and
      $|x|_\infty=\max\{|x_i|:1\le i\le d\}$
      For any $R\ge 0$,
      we define the box 

      $$
B_R=\{x\in\mathbb Z^d: |x|_\infty\le R\}.
      $$
 Let $b>0$ and $a>0$. We say a law $\mathbb P$
        on the environmental space satisfies {\it
          condition~$\BR{a}{b}$} if
        there exist a pair of numbers $R\ge 0$ and $c>0$ such that
	\begin{equation}
          \label{eq:moment}
	E_0 T_{B_R}^{a+c} < \infty, 
	\end{equation}
	and
	\begin{equation}\label{eq:randc}
    R> \frac{a(a+c)}{b\cdot c} -2.     
  \end{equation}

      \medskip 

      Singularities involving two orthogonal directions will play an important role throughout the article. Define 
\begin{equation}\label{def:eta*}
\eta_* := \max_{ i,j  \in \{1,\dots, 2d\} \text{ and } e_i \perp e_j} \{ \eta_i\wedge \eta_j\}. 
\end{equation}
To see why $\eta_{*}$  plays an important role in questions, write the singularities  in decreasing order $\eta_{j_1} \ge \eta_{j_2} \ge \eta_{j_3} \ge \dots \eta_{j_{2d}}$. And consider a situation in which we have $j_2 = -j_1$. Heuristically, it is the value of $\eta_* = \eta_{j_3}$ that would tell us to what extent the random walk is one dimensional:  $\eta_{j_3}$ close to zero means the transition probabilities are concentrated on $j_1$ and $-j_1$, 

\medskip 

\begin{theorem} \label{thm:BR}
  Consider an RWRE in an environment satisfying conditions $(E)_0$ and $(P)^{\ell}_M$ for some $M>15d+5$ and $\ell \in S^{d-1}$. If additionally the environment also satisfies $\BR{1}{\eta_*}$ then, the walk is ballistic in direction $\ell$. That is, $P_0$-almost surely, 
	$$
	\lim_{n\to \infty}\frac{X_n}{n} = v, \; \text{ where }v\cdot\ell >0, 
	$$
\end{theorem}
\medskip 

Roughly speaking, the above formal statement can be read in the following way:  under conditions $(E)_0$ and $(P)_M^{\ell}$, if either
the walk escapes a small ball really fast  (which corresponds to a
small $R$ and  a large moment for $T_{B_R}$), or the walk escapes in a
finite mean time a ball with large radius, we do have ballistic behavior. How large the radius has to
be, is determined by the  inequality \eqref{eq:randc}.

Moreover, up to arbitrarily small $\varepsilon$, condition $\BR{1}{\eta_*}$ is sharp. Indeed, for any given positive~$\varepsilon$, by taking $R$ large enough, condition $\BR{1}{\eta_*}$ becomes 
$$
E_0[T^{1+\varepsilon}_{B_R} ]< \infty. 
$$
And this can be contrasted with the proposition below  which gives zero speed behavior under~$E_0[T_{B_R}] = \infty.$

\medskip 

\begin{proposition}\label{prop:zerospeed} Consider an RWRE in an i.i.d elliptic environment. Also assume the walk is transient in direction $\ell$, for some $\ell \in \mathbb{S}^{d-1}$. Then, if  for some radius $R$ 
	$$
	E_0[ T_{B_R}] = \infty 
	$$
	the walk has zero speed. 
\end{proposition} 

\medskip

A useful corollary of Theorem 
      \ref{thm:BR}
      is the following. 

      \medskip 

      \begin{corollary}
        \label{corocor}
        Consider an RWRE satisfying $(E)_0$ and 
        $(P)_m|l$ for some 
        $M >15d+5$ and $l\in\mathbb S^{d-1}$. Assume that
        $\eta_*>1/2$. Then the walk is ballistic.
        \end{corollary}

        \medskip
Notice that the above corollary extends ballisiticity to a whole class
of elliptic environments. It says that uniform ellipticity can be
weakened and replaced by the  conditions $(E)_0$ and $\eta_* >
1/2$. That is, uniform ellipticity can be relaxed as long as we have
good enough behavior of the jump probabilities near zero on two orthogonal directions $e_j$ and $e_k$ such that $\omega(0,e_j)^{-1}$ and $\omega(0, e_k)^{-1}$ have light enough tails. 

Condition $\BR{1}{\eta_*}$ is  implied by the most general criteria for ballisticity for elliptic random walks in random environment. 
Fribergh and Kious proved in 
\cite{FK16} that under conditions $(E)_0$, $(P)_M^{\ell}$ for $M$
large enough and their condition $(K)_1$ the walk has ballistic
behavior. At Section \ref{sec:props} we prove that condition $(K)_1$ implies $\BR{1}{\eta_*}$.

\medskip

Under conditions which are stronger than those imposed in Theorems
\ref{thm:X}
and \ref{thm:BR}, we can derive central limit theorems. We say that 
an annealed central limit theorem is satisfied if 

	$$
	\epsilon^{1/2}\left(X_{\lfloor\epsilon^{-1}\cdot\rfloor}-\lfloor\epsilon^{-1}
          \cdot\rfloor v\right) 
	$$
        converges in law under $P_0$ as $\epsilon$ goes to $0$ to a 
        Brownian 
        Motion with non-degenerate deterministic covariance matrix.
        We say that 
a quenched central limit theorem is satisfied if $\mathbb P$-a.s.

	$$
	\epsilon^{1/2}\left(X_{[\epsilon^{-1}\cdot]}-[\epsilon^{-1}\cdot]v\right) 
	$$
        converges in law under $P_{0,\omega}$ as $\epsilon$ goes to $0$ to a 
        Brownian 
        Motion with non-degenerate deterministic covariance matrix.
        We have then the following annealed and quenched central limit theorems.

\medskip

\begin{theorem}\label{thm:cltX2}
  Consider an RWRE in $\Z^2$ whose
  environment satisfies conditions $(E)_0, (X)_2$ and $(P)^{\ell}_M$
  for some $M>35$ and $\ell \in S^{1}$.
  Then, both an annealed and a quenched central limit theorem are satisfied. 
      \end{theorem}

      \medskip

      \begin{theorem} \label{thm:cltBR}
        Consider an RWRE in an
              environment satisfying conditions $(E)_0$,
              $(P)^{\ell}_M$ for some $M>15d+5$ and $\ell \in
              S^{d-1}$ and               $\BR{2}{\eta_*}$.
                Then, both an annealed and a quenched central limit theorem are satisfied. 
              \end{theorem}

              \medskip

              We will continue with Section \ref{sec:localtrapping} where we will
              explain the meaning of condition $(X)_a$ and the
              necessity of introducing the correlation assumption (\ref{x:eq2}).
In Section \ref{section-two}  we will present the proof
of
theorems \ref{thm:X}, \ref{thm:BR}, \ref{thm:cltX2}
and \ref{thm:cltBR}. In Section \ref{sec:props} we will present the proof of Proposition \ref{prop:zerospeed} and a final discussion on the sharpness of our general condition $(B)_1^{\eta_*}$.

\medskip

\section{Local trapping and correlations}\label{sec:localtrapping}
In this section we discuss in detail condition $(X)_1$, more specifically \eqref{x:eq1} - \eqref{eq:squarerel}, together with the connection between singularities and local trapping. First let us explain the meaning behind \eqref{x:eq1} together with relations given by \eqref{eq:edgerel}-\eqref{eq:squarerel}. 
In what follows, we will call the exponents $\beta_\vdash$ and
$\alpha_\lrcorner$ and the exponents corresponding to rotations which
are multiples of $90$ degrees, the
{\it singularities} of the corresponding edges.

The most basic trap for the walk is a single edge. If we want to avoid
the walk to be trapped on it we should expect that, for each vertex on
the tip of the edge, the transition probabilities pointing out of the
edge have good tails, or in other words, have large singularities. This is schematically illustrated in the Figure 1.1.

We could reason in a similar manner for other structures more complex than an edge, such as wedges, horseshoes (which is pictured below) and squares. Thus, in general, one could argue that the walk should be able to escape any finite structure as long as the transitions of the `corners' of this structure have good enough singularities, with `good enough' meaning that the sum of the singularities is greater than one. This is illustrated in the picture below for a horseshoe format graph, and for a wedge and a square in Figure \ref{fig:wedgeblock}.

\begin{figure}[H]\label{fig:hscase}
	\includegraphics[scale=0.2]{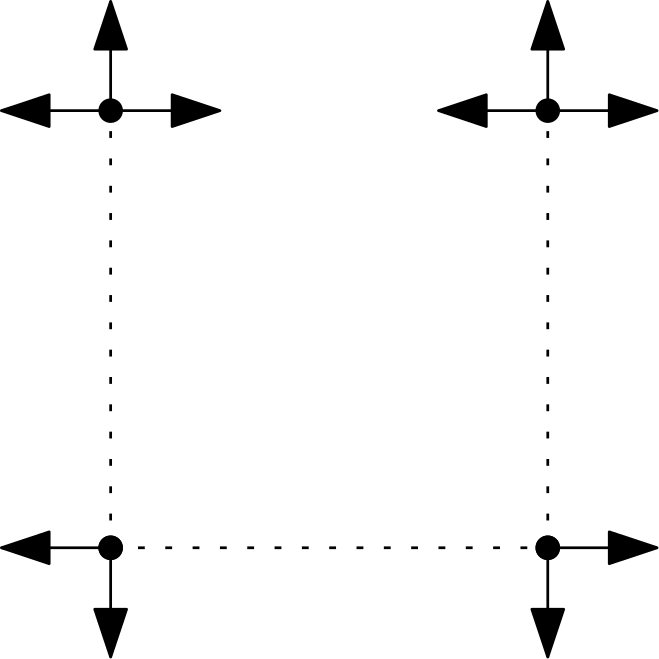}
	\caption{Schematic of a horseshoe and the transitions pointing out it.}
\end{figure}

In light of the above discussion, relation \eqref{x:eq1} together with relations \eqref{eq:edgerel}, \eqref{eq:wedgerel} and \eqref{eq:squarerel}, in condition $(X)_1$, prevent edges, wedges or squares whose transitions at the `corners' have bad singularities.  Notice that the case of a horseshoe is covered by the square, that is, if the transitions of the corners of a square have good tails, then the same holds for the transitions at the corners of a horseshoe. For this reason, condition $(X)_1$ does not include a relation covering specifically the singularities coming from a horseshoe.

 For the case of an edge $e$, in \cite{CR14} the authors prove that $E_0 \left[ \exittime{e}\right] $ is finite
if the singularities at the tip of the edge satisfy \eqref{eq:edgerel}. However, the reasoning of relating escapability to the singularities at the `corners' of a structure does not go much further. As we will show latter, it is possible to construct an environment such that the singularities of the `corners' of a wedge $W$ sum more than one, but the walk does not escape it in finite mean time, that is, $E_0 [\exittime{W}] = \infty$.

The above discussion together with Proposition \ref{prop:trapwedge} below show that the finiteness of $E_0 [\exittime{S}]$ for some finite graph $S$ other than a single edge hides correlations between the transitions on the vertices in $S$. In other words, we can say that in general the finiteness of $E_0 [\exittime{S}]$ cannot be guaranteed by a condition involving only the singularities of the transitions at the `corners' of $S$. For this reason, we have relations \eqref{x:eq2}, which should capture the correlations hidden by $E_0 [\exittime{S}] < \infty$. Here we must point out one of the advantages of condition $(X)_1$. Even though Proposition \ref{prop:trapwedge} shows that $E_0[\exittime{S}]$ involves correlations between the transitions on vertices in $S$, $(X)_1$ is still a condition which is verifiable by looking at the transitions of a single vertex. 

The remainder of this section is devoted to formalize the above discussion, that is, we construct an environment such that the singularities at the tips of edges, wedges and square sum more than one, but the walk still gets trapped in a wedge/square. In order to do that, consider the following densities
\begin{equation}
	f(x) = \begin{cases}
		C_1 x^ { \beta_{\dashv}-1}, & \text{ for }x \in (0,1/8] \\
		0, & \text{ otherwise.}
	\end{cases} \quad g(x) = \begin{cases}
		C_2 x^ { \beta_{\perp}-1}, & \text{ for }x \in (0,1/8] \\
		0, & \text{ otherwise.}
	\end{cases}
\end{equation}
and
\begin{equation}
	h(x) = \begin{cases}
		C_3 x^ { \beta_{\vdash}-1}, & \text{ for }x \in (0,1/8] \\
		0, & \text{ otherwise.}
	\end{cases},
\end{equation}
where $C_1,C_2,C_3$ are normalizing constants and $\beta_{\dashv}, \beta_{\perp}, \beta_{\vdash}$ are all strictly smaller than one and satisfy the following relations
\begin{equation}\label{eq:rel}
	\beta_{\dashv} \ge \beta_{\perp}; \quad 	\beta_{\dashv} + \beta_{\vdash} > 1 ; \quad \frac{\beta_{\dashv}}{2} + \beta_\perp + \beta_{\vdash} > 1; \quad \frac{\beta_{\dashv}}{2} + \frac{\beta_\perp}{2} + \beta_{\vdash} < 1.
\end{equation}
Now, consider the random variables $ \xi $, whose density is $f$, $\zeta$ whose density is $g$ and $\chi$ whose density is $h$. We then construct our environment in the following way: we consider the i.i.d sequences $\{\xi_x\}_{x \in \Z^2}$, $\{\zeta_x\}_{x \in \Z^2}$ and $\{\chi_x\}_{x \in \Z^2}$ together with an i.i.d sequence $\{U_x\}_{x \in \Z^2}$, where $U_x \sim Uni[0,1]$. We also assume that these four sequences are independent among themselves. Then, according to $U_x$ we assign one of the following transitions described below

\begin{figure}[H]\label{fig:env}
	\includegraphics[scale=1.25]{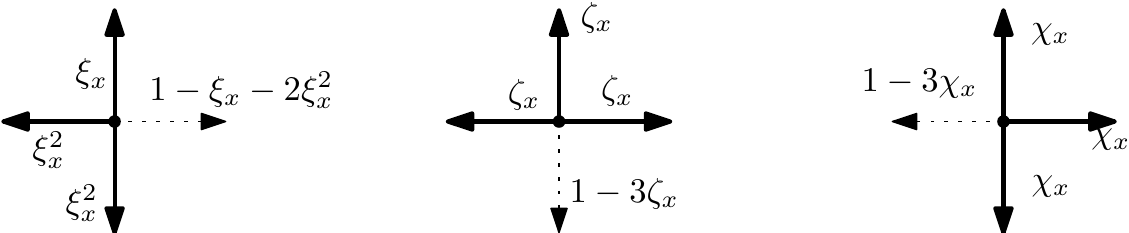}
	\caption{From left to right, three types of transitions: I, II and III.}
\end{figure}
More specifically, if $U_x \le 1/3$, we assign to $x$ a type I transition, if $1/3 \le U_x \le 2/3$ we assign to it a type II, whereas if $U_x \ge 2/3$ we assign a type III transition. 

Regarding the environment above defined, our first result concerns its singularities.
\begin{lemma}Consider a RWRE on $\Z^2$ with an i.i.d environment distributed as defined above. Then, it satisfies \eqref{x:eq1} and all relations given by \eqref{eq:edgerel}-\eqref{eq:squarerel}. However, it does not satisfies \eqref{x:eq2}.
	
\end{lemma}
\begin{proof} We begin with a technical comment. First observe that formally, condition \eqref{x:eq1} is not satisfied for the $\beta$'s exponents in the definition of the densities $f,g$ and $h$. However, since \eqref{x:eq1} is satisfied whenever we choose an exponent arbitrarily close to $\beta_\dashv$, for example, but smaller, we will abuse the notation by saying that $\beta_\dashv$ is exactly the same $\beta_\dashv$ in the definition of $f$. We also do the same thing for all the other singularities. 
	
	Observe that since $\xi,\zeta$ and $\chi$ are all smaller than $1/8$, the dashed directions illustrated by Figure 2.2 have probability at least $5/8$ to be crossed. This implies that if we consider the max of three directions $\dashv$ , either one of the three has probability at least $5/8$ to be crossed, what happens when we transitions of type II and III, or $Q_{\dashv}$ is distributed essentially as $\xi$, which implies that $Q_{\dashv}$ has singularity $\beta_{\dashv}$. Arguing similarly we conclude that $Q_\perp$ has singularity $\beta_\perp$, $Q_{\vdash}$ has singularity~$\beta_\vdash$. Moreover, using that $\beta_{\dashv} > \beta_\perp$, we also have that 
	\begin{equation}\label{eq:sings}
		\beta_{\top} = \infty,\; \alpha_{\stackrel{}{\ulcorner}} = \beta_{\vdash},\; \alpha_{\stackrel{}{\urcorner}} = \frac{\beta_{\dashv}}{2},\; \alpha_{\lrcorner} = \beta_{\perp},\; \alpha_{\llcorner} = \beta_\perp \wedge \beta_{\vdash}.
	\end{equation}
Notice that by \eqref{eq:rel} and the above relations, our singularities satisfies \eqref{eq:edgerel} and its 90 degree rotation, since $\beta_\top = \infty$. In the case of relations \eqref{eq:wedgerel} and its rotations, we do need to check the cases which include $\beta_\top$, since it is infinity. Thus, we are left to check 
$$
\alpha_{\stackrel{}{\urcorner}} + \beta_\perp + \beta_\vdash \stackrel{\eqref{eq:sings}}{=} \frac{\beta_{\dashv}}{2} + \beta_\perp + \beta_\vdash \stackrel{\eqref{eq:rel}}{>} 1,
$$
and
$$
\alpha_{\stackrel{}{\ulcorner}} + \beta_\perp + \beta_\dashv \stackrel{\eqref{eq:sings}}{=} \beta_\vdash + \beta_\perp + \beta_\dashv  \stackrel{\eqref{eq:rel}}{>} 1.
$$
For \eqref{eq:squarerel} we only have to check one condition since it is invariant under 90 degree rotations
$$
\alpha_{\stackrel{}{\ulcorner}}	+ \alpha_{\stackrel{}{\urcorner}} + \alpha_{\lrcorner}+ \alpha_{\llcorner} \stackrel{\eqref{eq:sings}}{=} \beta_{\vdash} + \frac{\beta_{\dashv}}{2} + \beta_{\perp} + \beta_\perp \wedge \beta_{\vdash} \stackrel{\eqref{eq:rel}}{>} 1,
$$
which proves that the environment satisfies \eqref{x:eq1} and all relations given by \eqref{eq:edgerel}-\eqref{eq:squarerel}. Notice that we have just proven that the structures edges, wedges and squares have the property that the sum of the singularities of the transition probabilities point out of them is greater than one. 

In order to prove that the environment does not satisfies \eqref{x:eq2}, notice that one of the requirement in such condition is given by
$$
\int_{\Omega} Q^{-\beta_{\perp}}_{\dashv} Q^{-\alpha_{\stackrel{}{\urcorner}}}_{\stackrel{}{\urcorner}} \mathbb{1}_{\{ \max\{\dashv\} = \uparrow\}} {\rm d}\mathbb{P}< \infty.
$$
Notice that the only transition type satisfying $\max\{\dashv\} = \uparrow$ is the type I. Thus, using the independence of $U_0$ and $\xi_0$ we have that
\begin{equation}\label{eq:infinity}
	\int_{\Omega} Q^{-\beta_{\perp}}_{\dashv} Q^{-\alpha_{\stackrel{}{\urcorner}}}_{\stackrel{}{\urcorner}} \mathbb{1}_{\{ \max\{\dashv\} = \uparrow\}} {\rm d}\mathbb{P} = \frac{1}{3}\int_{\Omega} \xi_0^{-\beta_\perp} \xi_0^{-2\cdot \frac{\beta_\dashv}{2}}{\rm d}\mathbb{P} =  \infty,	
\end{equation}
since $\xi$ has density $f$. However this is not enough to prove that the $(X)_1$ is not satisfied. Notice that condition $(X)_1$ requires the existence of a set of numbers $\alpha$'s and $\beta$'s satisfying relations \eqref{x:eq1}-\eqref{eq:squarerel}. At \eqref{eq:infinity} we showed that we cannot satisfy all relations required by $(X)_1$ choosing the largest $\alpha$'s and $\beta$'s. However we could try to choose a new set $\alpha'$'s and $\beta'$ with the property that for all directions $\alpha'\le \alpha$ and $\beta' \le \beta$. In the next lines we will show that is not possible to choose a set of $\alpha'$'s and $\beta'$'s that satisfies \eqref{x:eq1}-\eqref{eq:squarerel} under the additional constraint that  
\begin{equation}\label{eq:badrel}
\frac{\beta_{\dashv}}{2} + \frac{\beta_\perp}{2} + \beta_{\vdash} < 1.
\end{equation}
Notice that in order to satisfy \eqref{x:eq2}, we must have 
$$
\infty > \int_{\Omega} Q^{-\alpha'_{\stackrel{}{\lrcorner}}}_{\stackrel{}{\lrcorner}}Q^{-\alpha'_{\stackrel{}{\llcorner}}}_{\stackrel{}{\llcorner}}  {\rm d}\mathbb{P} > \frac{1}{3} \int_{\Omega} \zeta^{-\alpha'_{\stackrel{}{\lrcorner}}}\zeta^{-\alpha'_{\stackrel{}{\llcorner}}}  {\rm d}\mathbb{P},
$$
which implies that 
\begin{equation*}
	\alpha'_{\stackrel{}{\lrcorner}} + \alpha'_{\stackrel{}{\llcorner}} < \beta_\perp.
\end{equation*}
Arguing in a similar manner we have that
\begin{equation*}
	\infty > \int_{\Omega} Q^{-\alpha'_{\stackrel{}{\lrcorner}}}_{\stackrel{}{\lrcorner}}Q^{-\alpha'_{\stackrel{}{\urcorner}}}_{\stackrel{}{\urcorner}}  {\rm d}\mathbb{P} > \frac{1}{3} \int_{\Omega} \xi^{-\alpha'_{\stackrel{}{\lrcorner}}}\xi^{-2\alpha'_{\stackrel{}{\urcorner}}}  {\rm d}\mathbb{P},
\end{equation*}
which implies that
\begin{equation*}
	\alpha'_{\stackrel{}{\lrcorner}} + 2\alpha'_{\stackrel{}{\urcorner}} < \beta_\dashv.
\end{equation*}
And using the exact same reasoning we also deduce that 
\begin{equation*}
	\alpha'_{\stackrel{}{\llcorner}} + \alpha'_{\stackrel{}{\ulcorner}} < \beta_\vdash.
\end{equation*}
Using the above inequalities on \eqref{eq:badrel} leads us to
$$
1 > \frac{\beta_{\dashv}}{2} + \frac{\beta_\perp}{2} + \beta_{\vdash} \ge \frac{\alpha'_{\stackrel{}{\lrcorner}} + 2\alpha'_{\stackrel{}{\urcorner}}}{2} + \frac{\alpha'_{\stackrel{}{\lrcorner}} + \alpha'_{\stackrel{}{\llcorner}} }{2} + \alpha'_{\stackrel{}{\llcorner}} + \alpha'_{\stackrel{}{\ulcorner}} \ge \alpha'_{\stackrel{}{\lrcorner}} + \alpha'_{\stackrel{}{\urcorner}} + \alpha'_{\stackrel{}{\llcorner}} + \alpha'_{\stackrel{}{\ulcorner}},
$$
which contradicts \eqref{eq:squarerel}. Thus, under \eqref{eq:badrel}, we cannot choose exponents $\alpha'$'s and $\beta'$'s in order to satisfy $\alpha' \le \alpha, \beta'\le \beta$ and \eqref{x:eq1}-\eqref{eq:squarerel} together with \eqref{eq:badrel}. So such environment does not satisfy $(X)_1$.
\end{proof}
 Observe that by Lemma 2.1 in \cite{CR14}, the walk cannot be trapped in any edge. However, it can be trapped in a wedge/square, as ensures the proposition below.
\begin{proposition}[Trapped in a wedge/square]\label{prop:trapwedge}Consider a RWRE on $\Z^2$ with an i.i.d environment distributed as defined above. Let $W$ be the wedge defined by the following vertices~$0, (1,0)$ and $(0,1)$, then
	$$
		E_0 \left[ \exittime{W}\right] = \infty.
	$$
\end{proposition}
\begin{proof} Let $N_W(0)$ denote the number of visits to $0$ before leaving $W$. Clearly, we have that $\exittime{W} \ge N_W(0)$.
	On the other hand, under the quenched measure $P_{0,\omega}$, $N_W(0)$ can be written as $1 + Geo(P_{0,\omega}\left[\exittime{W} < H_0^+\right])$ where the geometric random variable is supported on $\{0,1,2,\dots\}$ and $H_0^+$ is the first return time to $0$. 

	Now, let $A$ be the environment in which we assign to $0$ a type I transition, to $(0,1)$ a type II and to $(1,0)$ a type III transition. Formally,
	$$
		A = \{ U_0 \le 1/3, U_{(1,0)} \ge 2/3, 1/3 \le U_{(0,1)} \le 2/3 \}
	$$
	Thus,
	\begin{equation}
		\mathbb{1}_A P_{0,\omega}\left[\exittime{W} < H_0^+\right] \le \mathbb{1}_A (2\xi^2_0 + 3\xi_0\zeta_{(0,1)} + 3\chi_{(1,0)} ),
	\end{equation}
	which implies that
	\begin{equation}\label{eq:b1}
		E_0 \left[N_W(0)\right] \ge \E \left[ \frac{1}{2\xi^2_0 + 3\xi_0\zeta_{(0,1)} + 3\chi_{(1,0)}} ; A \right] = \frac{1}{9} \int_{\Omega} \frac{1}{2\xi^2_0 + 3\xi_0\zeta_{(0,1)} + 3\chi_{(1,0)}} \mathrm{d}\P.
	\end{equation}
	Now, observe that 
	\begin{equation*}
		\P \left( \frac{1}{2\xi^2_0 + 3\xi_0\zeta_{(0,1)} + 3\chi_{(1,0)}} > u \right) = \P \left( 2\xi^2_0 + 3\xi_0\zeta_{(0,1)} + 3\chi_{(1,0)} < \frac{1}{u} \right). 
	\end{equation*}
	And by the independence of $\xi_0, \zeta_{(0,1)}$ and $\chi_{(1,0)}$ we have that  
	\begin{equation}\label{eq:lowerboundp}
		\begin{split}
			\P \left( 2\xi^2_0 + 3\xi_0\zeta_{(0,1)} + 3\chi_{(1,0)} < \frac{1}{u} \right) & \ge \P \left( 2\xi^2_0 < \frac{1}{3u}, 3\xi_0\zeta_{(0,1)} < \frac{1}{3u}, 3\chi_{(1,0)} < \frac{1}{3u} \right) \\
			& =  \P \left( \xi_0 < \frac{1}{\sqrt{6u}}, \xi_0\zeta_{(0,1)} < \frac{1}{9u} \right) \P \left(  \chi_{(1,0)} < \frac{1}{9u} \right).
		\end{split}
	\end{equation}
	Since $\chi_{(1,0)}$ has density $h$, there exists a positive constant $C_3'$ such that 
	\begin{equation}\label{eq:chi}
		\P \left(  \chi_{(1,0)} < \frac{1}{9u} \right) =  \frac{C_3'}{u^{\beta_{\vdash}}}.
	\end{equation}
	On the other hand, using the independence of $\xi_0$ and  $\zeta_{(0,1)}$ and that $\beta_{\dashv} > \beta_\perp$, there exist positive constants $C_1, C_1'$  and $C_2'$ such that
	\begin{equation}\label{eq:xi}
		\begin{split}
			\P \left( \xi_0 < \frac{1}{\sqrt{6u}}, \xi_0\zeta_{(0,1)} < \frac{1}{9u} \right) & = C_1\int_{0}^{1/\sqrt{6u}} \P \left( \zeta_{(0,1)} < \frac{1}{9ux} \right)x^{\beta_\dashv - 1}\mathrm{d}x \\ 
			& = C_1'u^{-\beta_\perp}\int_{0}^{1/\sqrt{6u}} x^{{\beta_\dashv - 1 - \beta_\perp}}\mathrm{d}x \\
			& =  C_2'u^{-\beta_\perp} \cdot u^{-(\beta_\dashv - \beta_\perp)/2} = \frac{C_2'}{u^{(\beta_\dashv + \beta_\perp)/2}}.
		\end{split}
	\end{equation}
	Finally, replacing \eqref{eq:chi} and \eqref{eq:xi} on \eqref{eq:lowerboundp} leads us to
	$$
	\int_{\Omega} \frac{1}{2\xi^2_0 + 3\xi_0\zeta_{(0,1)} + 3\chi_{(1,0)}} \mathrm{d}\P \ge \int_0^{\infty} \frac{C}{u^{(\beta_\dashv + \beta_\perp)/2 + \beta_{\vdash}}}\mathrm{d}u = \infty,
	$$
	since by \eqref{eq:rel} we have that $(\beta_\dashv + \beta_\perp)/2 + \beta_{\vdash} < 1$. Thus, by \eqref{eq:b1}, we prove the proposition.
\end{proof}
\medskip

\section{Proof of Theorems \ref{thm:X}, \ref{thm:BR}, \ref{thm:cltX2}
  and \ref{thm:cltBR}}
\label{section-two}
The first step towards the proof of Theorems
\ref{thm:X}, \ref{thm:BR}, \ref{thm:cltX2} and \ref{thm:cltBR},
is to reduce the proof to the task of obtaining good attainability estimates. Once this has been done, the rest of the argument is to prove that the local conditions $(E)_0$ and $\BR{a}{\eta_*}$ imply that the walk is capable of escaping growing regions of $\Z^d$ fast enough.

\medskip

\subsection{Attainability estimate}
Here  we make precise what is meant by an environment to have good attainability. 
For any subset $A\subset\mathbb Z^d$, we define the exit time 
of $A$ by 

$$
\exittime{A}:=\inf\{n\ge 0: X_n\notin A\}.
$$
Furthermore, we define the hitting times

$$
\hittime{A}:=\inf\{n\ge 1: X_n\in A\} 
$$
and

$$
\hittime{A}^+:=\inf\{n\ge 1: X_n\in A\}. 
$$

\medskip 

\begin{definition}[$b$-good attainability] Let $b>0$. We say a random environment on $\Z^d$ has $b$-\textit{good attainability} and denote it by $(A)_b$ if there exists $\varepsilon>0$ such that for all~$\delta>0$ there is a~$\delta'>0$ and $u_0$ such that, for 
  all $u\ge u_0$ we have that

	\begin{equation}\label{def:attain}
	\P\left( \max_{y: |y|= \delta' \log u}P_{0, \omega}\left( \hittime{y} < \hittime{0}^+\right) \le u^{-\frac{b+2\delta}{b+\varepsilon}}\right) \le \frac{1}{u^{b+\delta}}. 
	\end{equation}
      \end{definition}

      \medskip 
      
      Notice that the above condition is not local in nature, since it involves escaping a ball whose radius is going to infinity. In what follows we recall the connection between upper bound on the tail of the first regeneration time $\tau_1$ and $(A)_a$. To 
      do this we will first define the concept of regeneration times. 
      Let $(\mathcal F_n)_{n\ge 0}$ be the natural filtration of the random walk and 
      $(\theta_n)_{n\ge 0}$ the canonical shift in $(\mathbb Z^d)^{\mathbb N}$. 
      Let $l\in\mathbb S^{d-1}$ and $a>0$. Define 

      $$
\bar T_a=\min\{k\ge 1: X_k\cdot l\ge a\}
$$
and 

$$
D=\min\{m\ge 0: X_m\cdot l<X_0\cdot l\}. 
$$
We now define two sequences of $\mathcal F_n$-stopping times $(S_n)_{n\ge 0}$
and $(D_n)_{n\ge 0}$. Let $S_0=0$, $R_0=X_0\cdot l$ and $D_0=0$. 
Now, define by induction in $k\ge 0$, 

\begin{eqnarray*}
&S_{k+1}=\bar T_{R_k+1},\\
  & D_{k+1}=D\circ\theta_{S_{k+1}}+S_{k+1}\\
  & R_{k+1}=\sup\{X_i\cdot l:0\le i\le D_{k+1}\}. 
\end{eqnarray*}
Let 

$$
K=\inf\{n\ge 0: S_n<\infty, D_n=\infty\}
$$
with the convention that $K=\infty$ when $\{n:S_n<\infty, D_n=\infty\}=\emptyset$. 
We define the first regeneration time by 

$$
\tau_1=S_K. 
$$
Observe that the bound provided in the theorem below is as good as the
one given by the attainability property. The following result, which
corresponds to Proposition 5.1 of \cite{FK16} (see also \cite{CR14}),
shows
how
an attainability estimate provides bounds on the tails of the
first regeneration time.

\medskip 
\begin{theorem}\label{thm:framework} Consider an RWRE satisfying in an environment conditions $(E)_0$, $(A)_b$, $(P)_M^\ell$ for some $M>15d+5$, $b>0$ and $\ell \in S^{d-1}$ . Then, there exist $\delta>0$ and $u_0>0$ such that for  $u\ge u_0$, 
	$$
	P_0\left( \tau_1 > u \right) \le u^{-(b+\delta)}. 
	$$
      \end{theorem}

      \medskip
      A combination of the above result with Theorem 1.1 in
      \cite{CR14}, shows through the following theorem, the key role played by attainability
      estimates to prove the law of large numbers  and central limit theorems. 

      \medskip
      
\begin{theorem}\label{thm:genframework} Consider an RWRE satisfying in
  an environment conditions $(E)_0$ and
  $(P)_M^\ell$ for some $M>15d+5$ and $\ell \in S^{d-1}$ . Then, 
	\begin{enumerate}[label=(\alph*)]
		\item if $(A)_1$ is satisfied, there exist a deterministic $v \neq 0$ such that 
		$$
		\lim_{n\to \infty} \frac{X_n}{n} = v. 
		$$
              \item  if $(A)_2$ is satisfied, then the random walk
                satisfies
                  both an annealed and a quenched
                  central
                  limit theorem.
	\end{enumerate}
\end{theorem}

\medskip

\subsection{Proof of Theorems \ref{thm:BR} and \ref{thm:cltBR}}
In the light of Theorem \ref{thm:genframework},
in order to prove  Theorem \ref{thm:BR} (respec.
 Theorem \ref{thm:cltBR}) it is enough to show that under $(E)_0$ and
 condition $\BR{1}{\eta_*}$ (respec $\BR{2}{\eta_*}$) condition
 $(A)_1$ (respec. $(A)_2$) holds. However, instead of proving it
 directly, we will take a step back and prove a more general
 result. We will prove $a$-good attainability under $\BR{1}{\eta_*}$
 and a general condition~$\conditionH$ and then prove that $(E)_0$ implies $\conditionH$.

 Before we define $\conditionH$, we recall some standard notation.
For each $R>0$, we define 
\begin{equation}
  \left\lbrace 0 \longrightarrow \partial B_R \right \rbrace := \left\lbrace \exittime{B_{R-1}} < \hittime{0}^+\right \rbrace, 
  \end{equation}
that is, the event that the walk hits $\partial B_R$ before returning to the origin.
 For
a fixed $e_i$ in the canonical basis,
write $\cal{V}_i := <e_i>^{\perp}$, that is, the hyperplane orthogonal to $e_i$. Also let 
\begin{equation}
\left\lbrace 0 \stackrel{\cal{V}_i}{\longrightarrow} \partial B_R \right \rbrace := \left\lbrace \hittime{\partial B_R} < \hittime{0}^+\wedge \exittime{\cal{V}_i}\right \rbrace, 
\end{equation}
that is, the event in which the walk hits $\partial B_R$ before
returning to the origin without leaving $\mathcal{V}_i$.

\medskip
\begin{definition}[Condition $\conditionH$]
We say that an RWRE satisfies condition $\conditionH$ if, for each
direction $e_i$ there exist positive constants $C_i$ and $\widetilde{\eta}_i$,
such that for all $q \in [0,1]$ and $R \in \mathbb{N}$ one has that
\begin{equation}\label{eq:H}
\P\left( P_{0,\omega} \left(0 \stackrel{\cal{V}_i}{\longrightarrow} \partial B_R \right) \le q \right) \le q^{\widetilde{\eta}_i}C_i^R.
\end{equation}
\end{definition}

\medskip

 Notice that it is enough for an environment to have only $2$
perpendicular
{\it good} directions in order to satisfy $\conditionH$, in the sense that, it is enough to have two orthogonal directions~$e_i$ and~$e_j$ and two positive constants $\widetilde{\eta}_i$ and $\widetilde{\eta}_j$ such that 
$$
\E \left[ \omega(0,e_i)^{-\widetilde{\eta}_i}\right] \vee \E \left[ \omega(0,e_j)^{-\widetilde{\eta}_j}\right] < \infty.
$$
Hence, an environment does not need to  satisfy $(E)_0$, not
even it has to be elliptic, in order for it to satisfy this condition. 

Our next result shows us how we can combine condition $\conditionH$ with some moment condition on $T_{B_R}$ in order to guarantee good attainability.
But before we state it, we will need an intermediate step.

\medskip

\begin{lemma}\label{lemma:aux} Consider an RWRE on $\Z^d$ satisfying
  condition $\BR{a}{b}$, for $a\ge 1$ and $b>0$. Then, there exists a
  constant $C$ depending on $a, b, c$ and $R$ (where $R$ and $c$ are
  the constants of the definition of $\BR{a}{b}$ so that $R$ 
   satisfies inequality (\ref{eq:randc}) involving also
  to $a,b$ and $c$) such that for $u\ge 1$
  $$
  \P\left(\Pq{0}\left(0 \rightarrow \partial B_{R+1} \right)\le u^{-1}\right) =  \P \left( \Pq{0}\left( \exittime{B_R} < \hittime{0}^+\right) \le u^{-1}\right) \le Cu^{-a-c}.
  $$
\end{lemma}
\begin{proof} Let $R$ and $c$ in condition $\BR{a}{b}$ be fixed and
  denote by
  $N_{B_R}(0)$ the number of returns to the origin before leaving $B_R$. 
Observe that $T_{B_R}$ is greater than $N_{B_R}(0) $ almost
surely. Moreover, by
the strong Markov property it follows that $N_{B_R}(0)$ has the same
law as a geometric random variable of parameter $\Pq{0}\left( \exittime{B_R} < \hittime{0}^+\right)$ supported on $\{0,1,\dots\}$, under the quenched measure $\Pq{0}$.
Combining the above discussed with condition $\BR{a}{b}$ and Jensen's inequality
\begin{equation*}
  \begin{split}
    \E \left[\left(\frac{1-\Pq{0}\left( \exittime{B_R} < \hittime{0}^+\right)}{\Pq{0}\left( \exittime{B_R} < \hittime{0}^+\right)} \right)^{a+c}\right] & = \E\left[\left(E_{0,\omega} N_{B_R}(0)\right)^{a+c}\right]  \le E_0 N_{B_R}^{a+c}(0) \le E_0 T_{B_R}^{a+c} < \infty.
  \end{split}
\end{equation*}
From the above inequality it follows that 
\begin{eqnarray*}
&\E \left[\left(\frac{1}{\Pq{0}\left( \exittime{B_R} <
        \hittime{0}^+\right)
    } \right)^{a+c}
\right]\\
&=
\E \left[\left(\frac{1}{\Pq{0}\left( \exittime{B_R} <
            \hittime{0}^+\right)} \right)^{a+c}, 
            \Pq{0}\left( \exittime{B_R} <
        \hittime{0}^+\right)\le 1/2 
\right]\\
&+
\E \left[\left(\frac{1}{\Pq{0}\left( \exittime{B_R} <
        \hittime{0}^+\right)} \right)^{a+c},\Pq{0}\left( \exittime{B_R} <
            \hittime{0}^+\right)> 1/2\right]\\
  &\le
    2^{a+c}\E \left[\left(\frac{1-
\Pq{0}\left( \exittime{B_R} <
    \hittime{0}^+\right)
    }{\Pq{0}\left( \exittime{B_R} <
    \hittime{0}^+\right)
    } \right)^{a+c}, 
            \Pq{0}\left( \exittime{B_R} <
        \hittime{0}^+\right)\le 1/2 
\right]+2^{a+c}<\infty,
\end{eqnarray*}
which combined with Markov Inequality proves the lemma.

\end{proof}

Now we can prove the following proposition.

\begin{proposition}\label{prop:BR} Consider an RWRE on $\Z^d$
  satisfying condition $\conditionH$. Additionally let $a \ge 1$ and~$b=\min\{\widetilde{\eta}_i:1\le i\le d\}$ and assume that condition $(B)_a^{b}$ is satisfied.
   Then, there exist $\delta$ and $\varepsilon$, such that
\begin{equation}
\P\left(\max_{y \in \partial B_{\delta\log u}} \Pq{o}\left( \hittime{y} < \hittime{x}^+\right) \le u^{-1}\right) \le u^{-a-\varepsilon},
\end{equation}
for all $u$ sufficiently large. In words, under $\conditionH$ and $(B)_a^b$, the walk has $a$-good attainability.
\end{proposition}

\medskip
Before we prove the result, let us say some words about its statement
and why it is important. The above proposition says that under
$\conditionH$, in order to guarantee that the walk is capable of
reaching distance
$\delta\log u$ with a    high enough probability, it is enough to
analyze its
behavior inside a ball of radius $R$. Observe that \eqref{eq:randc}
gives some sort of trade-off to check \eqref{eq:moment}. If we want to
check \eqref{eq:moment} for a small $c$,  then we need to consider a
large radius $R$. On the other hand, if we want to obtain a condition
verifiable on a small box, then we must guarantee that the walk
escapes this small box fast enough, i.e., $\escball{R}$ has high
$P_0$-moments.

\begin{proof} Let us explain the idea of the argument which is similar
  to some
  methods that were already used in \cite{CR14}. We first
  guarantee that with high probability, $B_R$  will be crossed in all
  directions by \textit{good} hyperplanes.
  In this case, \textit{good} means that it  will not be too costly
  in terms of probability, for the  walk to go through these
  hyperplanes. Then, \eqref{eq:moment} guarantees
  that there exists a \textit{good} path going from the origin to the
  boundary of $B_{R+1}$.
  Thus we can use this path to reach some good hyperplane that leads
  us to the boundary of the larger box
  $B_{\delta \log u}$. The picture below is an illustration of the
  above strategy for the case $d=2$. 
	\begin{figure}[h]
		\centering 
		\includegraphics[width=0.7\linewidth]{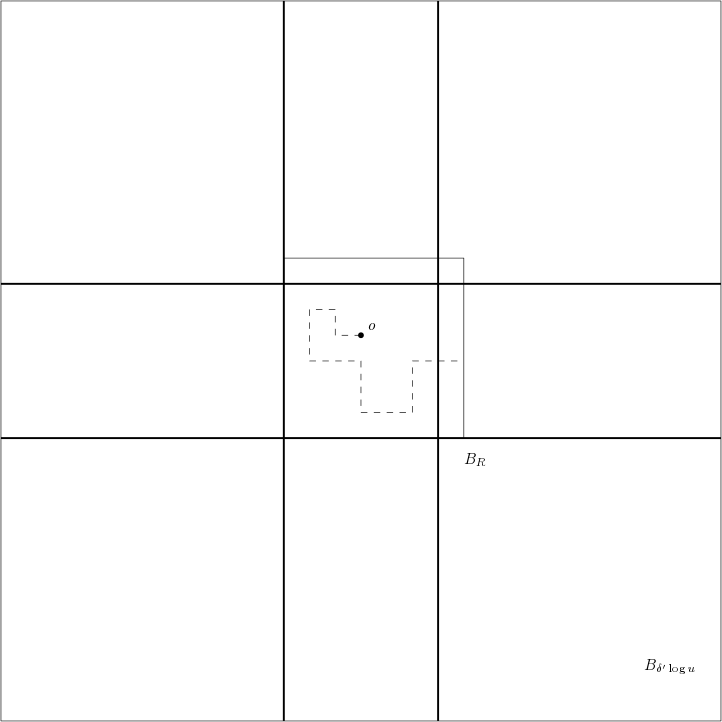}
		\caption{Good hyperplanes (strong lines) crossing the ball $B_R$ and a good path (dashed) from $o$ to $\partial B_R$}
		\label{fig:circuit}
	\end{figure}

	Fix $e_i$ in the canonical basis. Observe that $|\mathcal{V}_i
        \cap B_R| = (2\lfloor R\rfloor +1)^{d-1}$.  Let
        $\delta,\delta'>0$. We will say that a point of
        $x \in \mathcal{V}_i \cap B_{R+1}$  is $(\delta,\delta')$-\textit{bad} if, for a
        small $\delta'$ to be chosen latter

	\[
	P_{x,\omega} \left( x \stackrel{\mathcal{V}_i}{\longrightarrow} \partial B_{\delta\log u}(x) \right) \le u^{-\delta'}. 
      \]
	We will also say that the hyperplane $\mathcal{V}_i$ is $(\delta,\delta')$-\textit{bad} if there
        is some $x \in \mathcal{V}_i \cap B_{R+1}$ which is $(\delta,\delta')$-bad. Thus,
        using the fact that the environment is i.i.d., condition
        $\conditionH$ for direction $e_i$ and the union bound, we have 
	\begin{equation}\label{ineq:badplane}
	\P\left( \mathcal{V}_i \text{ is bad }\right) \le (2(R+1))^{d-1}\P\left( P_{0,\omega} \left( 0 \stackrel{\mathcal{V}_i}{\longrightarrow} \partial B_{\delta\log u} \right) \le u^{-\delta'}\right) \le \frac{(2(R+1))^{d-1}C_i^{\delta\log u}}{u^{\delta'\widetilde{\eta}_i}}.
	\end{equation}
	Finally, we say that direction $e_i$, $1\le i\le 2d$, is $(\delta,\delta')$-\textit{bad} if
        $\mathcal{V}_i +me_i$ is $(\delta,\delta')$-bad for all $m \in \{0,\dots,
        R+1\}$. Using again the fact that the environment is i.i.d. we
        see that

	\begin{equation}\label{ineq:baddirection}
	\P\left( \text{direction }e_i \text{ is bad }\right) = \P\left( \bigcap_{m=0}^{R+1}\{\mathcal{V}_i + me_i \text{ is bad} \}\right) = \P\left( \mathcal{V}_i \text{ is bad }\right)^{R+2}. 
	\end{equation}
	Now observe, from Equation \eqref{ineq:badplane}, that by setting 
	\begin{equation}\label{def:delta}
	\delta = \frac{\delta'}{2}\min_i\left\lbrace \widetilde{\eta}_i\right\rbrace / \max_i \{\log C_i \},
      \end{equation}
we see that for any direction $i$ 
	$$
	\P\left( \mathcal{V}_i \text{ is bad }\right) \le \frac{C_{R,d}}{u^{\delta'\min_i\left\lbrace \widetilde{\eta}_i\right\rbrace/2}}, 
	$$
	 where $C_{R,d}$ is a positive constant depending on $R$ and
         the dimension $d$ only.  Notice that we tacitly assumed $\log
         C_i$ is positive for all $C_i$ in \eqref{def:delta}. This is
         possible because we can assume $C_i >1$, since this only
         makes \eqref{ineq:badplane} worse. Thus, returning to
         \eqref{ineq:baddirection}  and recalling that $b = \min_i \widetilde{\eta}_i$, we have 
	\begin{equation}\label{eq:baddirect2}
	\P\left( \text{direction }e_i \text{ is bad }\right) \le
        \frac{C^{R+2}_{R,d}}{u^{\delta'b(R+2)/2}} \le \frac{1}{u^{a+\varepsilon'}}, 
	\end{equation}
	for some $\varepsilon'$, provided $u$ is large enough and $\delta'b(R+2)/2 >a$ and $a\ge 1$. 
        Now,  condition~$\BR{a}{b}$ and Lemma \ref{lemma:aux} leads us to, 
	\begin{equation}\label{ineq:pathK}
	\P\left(P_{0,\omega}\left(0 \rightarrow \partial B_{R+1} \right) \le \frac{1}{u^{1-c/K}}\right) \le \frac{1}{u^{(1-c/K)(a+c)}} = \frac{1}{u^{a+\varepsilon}}
	\end{equation}
	whenever $K>a+c$ and $u$ is large enough.
        Now, we choose some $\delta'$ such that 
	\begin{equation}
	\frac{a}{b(R+2)} < \delta' < \frac{c}{a+c}, 
      \end{equation}
    	and $K$ such that $ a+c<K<  c/\delta'$. These choices
        of $K$ and $\delta'$
        are possible due to \eqref{eq:randc}.

	Now, define the events
	\[
	A_1 := \{\text{ all the 2\textit{d} directions are}\ (\delta,\delta')-\text{good}\}
	\]
	and 
	\[
	A_2 := \left \lbrace P_{0,\omega}\left( 0 \rightarrow \partial B_{R+1} \right) \ge \frac{1}{u^{1-c/K}}\right \rbrace. 
	\]
	Note that 
	\begin{equation}
	\P\left( P_{0, \omega}\left(0 \rightarrow \partial B_{\delta\log (u)-R}\right) \le u^{-1}, A_1, A_2 \right) = 0, 
	\end{equation}
	since the probability of going from $0$ to some good (affine)
        hyperplane is at least $1/u^{1-c/K}$ and the probability of
        going from the hyperplane to $\partial B_{\delta \log (u)-R}$
        is at least $1/u^{\delta'}$, but recall that we have chosen $\delta'$ in such way that $\delta' < c/K$. Moreover, by \eqref{eq:baddirect2} and \eqref{ineq:pathK} we have 
	\[
	\P\left(A_1^c \cup A_2^c \right) \le \frac{1}{u^{a+\varepsilon''}}, 
	\]
	for large enough $u$ and some positive $\varepsilon''$. By
        intersecting $\{P_{0, \omega}\left( 0 \rightarrow \partial
          B_{\delta\log(u) -R}\right) \le u^{-1}\}$ with the event
        $A_1\cap A_2$ and its complement $A_1^c\cup A_2^c$ we prove the proposition.         \end{proof}
\medskip

The next lemma guarantees that $(E_0)$ implies $\conditionH$. 
\begin{lemma}\label{lemma:mineta}Consider a random environment on
  $\Z^d$ satisfying condition $(E)_0$. Then it satisfies condition
  $\conditionH$  in a way that $\min_i \widetilde{\eta}_i
    \ge \eta_{*}$, where $\eta_{*}$ is defined
    in~\eqref{def:eta*}. 
\end{lemma}
\begin{proof} We want to prove that for each direction $e_i$ there exist positive constants $C_i$ and~$\widetilde{\eta}_i$ such that for all $q \in [0,1]$ and $R \in \N$
  \begin{equation}\label{eq:h}
    \P\left( P_{0,\omega}\left(0 \stackrel{\mathcal{V}_i}{\rightarrow}\partial B_R\right)\right) \le q^{\widetilde{\eta}_i}C_i^R.
  \end{equation}
  Additionally, we also want that $\min_i \widetilde{\eta}_i \ge \eta_*$. In this direction, observe that if $e_m$ is orthogonal to $e_i$ then we can go from $0$ to $\partial B_R$ by taking $R$ steps only at direction $e_m$. Since we are under condition $(E)_0$, by Markov inequality we have that
  \[
    \P\left( \prod_{k=0}^{R-1}\omega(ke_{m},e_{m}) \le q \right) \le q^{\eta_{m}}\E\left[\omega(0,e_{m})^{-\eta_{m}}\right]^R.
  \]
  However, in order to maximize the value of $\widetilde{\eta}_i$ in \eqref{eq:h} and to ensure that $\min_i \widetilde{\eta}_i \ge \eta_*$, we must choose the direction $e_m$ properly. In order to do so, we will consider the worst scenario for our choices which corresponds to that one whose two directions with largest singularities are not perpendicular to each other.

  Thus, suppose the two largest values among
  $\eta_1,\ldots,\eta_{2d}$ on condition $(E)_0$ correspond to
  directions $j$ and $-j$. Let $i_0$ be the direction (orthogonal to
  $j$ and $-j$) such that $\eta_{i_0}$ is the third largest singularity. 
	For a fixed direction $i$, we proceed as follows: if  either~$i=-j$ or
        $i=j$,
we then have that $e_{i_0} \in \cal{V}_i$. Now, consider the line segment from $0$ to $\partial B_R$ in the direction $e_{i_0}$. Then, Markov's inequality and $(E)_0$ yield 
\begin{equation}\label{ineq:EimpliesH}
\P\left( P_{0,\omega} \left(0 \stackrel{\cal{V}_i}{\longrightarrow} \partial B_R \right) \le q \right) \le \P\left( \prod_{k=0}^{R-1}\omega(ke_{i_0},e_{i_0}) \le q \right) \le q^{\eta_{i_0}}\E\left[\omega(0,e_{i_0})^{-\eta_{i_0}}\right]^R<\infty.
\end{equation}
On the other hand, if $j \in \mathcal{V}_i$, then we hit $\partial B_R$ going straight to it using direction $e_j$ and repeat the above bound using $e_j$. Thus, condition $\conditionH$ is satisfied in a way that either $\widetilde{\eta}_i = \eta_j \ge \eta_*$ or $\widetilde{\eta}_i = \eta_{i_0} = \eta_*$, which proves the lemma. 
\end{proof}
Now we have all the results needed to prove the general positive speed criteria (Theorem \ref{thm:BR})and the central limit theorem (Theorem \ref{thm:cltBR}). 
\begin{proof}[Proof of Theorems \ref{thm:BR} and \ref{thm:cltBR}] The
  proof of both theorems is a matter of putting together the results
  we have developed so far. From Lemma \ref{lemma:mineta} we have that
  under $(E)_0$, condition $\conditionH$ is satisfied in a way that
  $\min_i \eta_i \ge \eta_{*}$, where $\eta_{*}$ is the third largest
  singularity given by $(E)_0$. Moreover, under the hypothesis of Theorem
  \ref{thm:BR}, by Proposition \ref{prop:BR} the walk has $1$-good
  attainability. Thus, Theorems \ref{thm:framework} and \ref{thm:genframework} imply ballisticity. 	On the other hand, under the hypothesis of Theorem \ref{thm:cltBR} we have $2$-good attainability which is enough to prove Theorem \ref{thm:cltBR}. 
\end{proof}

\medskip

We end this section showing how Theorem \ref{thm:BR} implies Corollary \ref{corocor}
\begin{proof}[Proof of Corollary \ref{corocor}
  ] Observe when $\eta_*> 1/2$ there exists $c_*>0$ such that 
	$$
	\frac{1+c_*}{\eta_* c_*} -2 < 0. 
	$$
	Thus, choosing $R= 0$ and noticing that $T_{B_0} = 1$, $P_o$-a.s. it follows that condition 
	$$
		E_oT_{B_0}^{1+c_*} <\infty 
	$$
	is trivially satisfied. Applying Theorem \ref{thm:BR} we prove the result. 
      \end{proof}

      \medskip

      \subsection{Proof of Theorems \ref{thm:X} and \ref{thm:cltX2}}
      To prove the computable criteria theorem for $\Z^2$ we will use
      Theorems \ref{thm:BR} and \ref{thm:cltBR}. In  light of both
      theorems, instead of proving that the walk escapes a growing
      region of $\Z^2$ (that is, the attainability condition $(A)_a$), we can reduce the work to prove that the walk escapes fast enough a finite region, i.e., the ball $B_R$. 
 
In order to guarantee that a RWRE under $(X)_1$ escapes any $B_R$ in finite mean time we will introduce the concept of exit strategy, which will help us to bound the probability of reaching $\partial B_R$. 

Since theses ideas rely on the language of flow networks, we will
introduce the main definitions and results about flows in the next
subsection. Then, we will prove how, in our context of RWRE, flows may
be useful to bound paths probabilities on a finite ball. In
Section \ref{sss1},  we will review some results about
flows on directed graphs. In Section \ref{sss2} we will show
how the theory of flows can be used to obtain bound on atypically
small probabilities and define a random graph process on
$B_R$. Finally, in
Section \ref{sss3}, we will prove  Theorem \ref{thm:X} by proving that there exist a random flow having good properties supported on the graphs generated by our graph process.

\subsubsection{Some results about flows on directed graphs}
\label{sss1}
Our techniques to prove Theorems \ref{thm:X} and \ref{thm:cltX2} rely
on flows over directed graphs. For this reason, we will introduce some
definitions and important results on the subject here.  The reader can
also
consult the textbook \cite{LP17} for more details.

A directed graph $G = (V, \mathcal{E})$ is a graph whose edges have a direction. For an edge~$e = (e_{-}, e_+) \in \mathcal{E}$, we call the vertex $e_{-}$ the \emph{tail} of $e$ and $e_+$ the \emph{head} of $e$. Thus the edge $e$ goes from $e_-$ to~$e_+$. Given a (un)directed graph $G$, we will denote its edge set by~$\mathcal{E}(G)$, or simply $\mathcal{E}$ when $G$ is clear from the context.

\medskip
\begin{definition}[Flow] Consider a directed graph $G=(V,\mathcal{E})$.
     A flow $\theta$ on $G$ with \emph{source} $A \subset V$ and \emph{sink}
     $Z \subset V$  on $G$ is a function $\theta : \mathcal{E} \to
     \R_+$ satisfying the following conditions,
 
	\begin{enumerate}
		\item For all $x \in (A \cup Z)^c $, 
		\[
		{\rm div}\theta(x) := \sum_{e \in \mathcal{E}, \, e_- = x}\theta(e) - \sum_{e \in \mathcal{E}, \, e_+ = x}\theta(e) = 0; 
		\]
		\item For all $x \in A$, ${\rm div}\theta(x) \ge 0$; 
		\item For all $x \in Z$, ${\rm div}\theta(x) \le 0$. 
	\end{enumerate}

      \end{definition}

      \medskip
The \textit{strength} of a flow $\theta$ is the total amount of flow going from the source to the sink and will be denoted by 
\begin{equation}
\|\theta\| := 	\sum_{x \in A} {\rm div}\theta(x). 
\end{equation}

   A \textit{non-degenerate} flow from $A$ to $Z$  is a flow  with
  source $A$ and sink $Z$ and $\| \theta \| >0$. A \textit{unit flow}
is a flow  of strength $1$. A \textit{capacity} function on a directed graph is a function $c: \mathcal{E}(G) \to \mathbb{R}_+$. A directed graph together with a capacity is called a \textit{network}. We will call a flow $\theta$ on a network $G$ \textit{admissible} if it satisfies $\theta(e) \le c(e)$ for all $e \in \mathcal{E}(G)$. In words, if the flow does not exceed edges' capacity.

In the context of RWRE, we let flows and capacities depend on the
environment configuration $\omega$. Thus a \textit{random flow}
$\theta$ on $\Z^d$ from $A \subset \Z^d$ to $Z\subset \Z^d$ is a
function~$\theta: \Omega \times \mathcal{E}(\Z^d) \to \mathbb{R}_+$
such that $\theta(\omega, \cdot)$ is a flow on the graph
$(\Z^d,\mathbb E^d)$ (where $\mathbb E^d$ are the directed nearest
neighbor edges of $\Z^d$) from $A$ to $Z$ for almost every
$\omega$. 
Under such definitions, the strength of a random flow $\theta$ is a random variable on $\Omega$. 
We similarly define \textit{random} capacity.

We say that a subset of edges $\Pi$ \textit{separates} $A$ from $Z$ if all paths going from $A$ to $Z$ use at least one edge in $\Pi$.  In this case we say $\Pi$ is a \textit{cutset}. 

It will be useful for our purposes to construct flows satisfying some
constrains. For this purpose we will use the following generalized version of
the classical Max-flow  Min-cut theorem.

\medskip

\begin{theorem}[Max-Flow Min-Cut Theorem, \cite{LP17}]\label{th:maxflow} Let $A$ and $Z$ be disjoint sets of vertices in a directed finite network $G$. The maximum strength of an admissible flow between $A$ and $Z$ equals the minimum cutset sum of the capacities. In symbols, 
	\begin{equation}
	\begin{split}
	\max \Lbrace \|\theta \|; \; \theta \text{ is an admissible
          flow from }A\text{ to } Z  \text{satisfying } \forall e\
        0\le\theta(e)\le c(e)\Rbrace\\ 
	=  \min \Lbrace \sum_{e \in \Pi} c(e), \Pi \text{ separates }A \text{ from }Z \Rbrace. 
	\end{split}
	\end{equation}
\end{theorem}

\medskip
Observe that a undirected $G$ graph may be transformed into a directed
one by duplicating every edge of $G$ and considering two edges, one
for each direction. We call the directed graph obtained from this operation the \textit{directed} version of $G$.

\subsubsection{Flows and probability of paths}
\label{sss2} In this part we will show how a flow can be used to bound the atypically small
probabilities of escaping a ball $B_R$. Our main result in this part is Lemma \ref{lemma:flowpath}, but before we state and prove it, we will need additional terminology as well as an intermediate result.

 It will be useful to our purposes to decompose a given (random) flow from $0$ to $\partial B_R$ on the directed version of $B_R$, as a finite collection of directed weighted paths going from $0$ to $\partial B_R$. For a given path $\sigma$ in this decomposition, we let $p_\sigma$ be its weight and we will write $\sigma \in \theta$ to mean that $\sigma$ is a directed path from $0$ to $\partial B_R$ such that $p_\sigma > 0$.

 The lemma below guarantees this decomposition and connects the $p$-weights assigned to paths with the strength of a random flow. It states that the amount of flow flowing from $0$ to $\partial B_R$ is the sum of the $p$-weights over the directed paths from $0$ to $\partial B_R$.

\begin{lemma}\label{lemma:pweights} For a fixed positive integer $R$, let $\theta$ be a (random) flow from $0$ to $\partial B_R$ supported on the directed version of $B_R$ such that $\|\theta \| > 0$, $\P$-almost surely. Then, we can assign weights to directed paths from $0$ to $\partial B_R$ in a way that
  \begin{equation}\label{eq:pweights_strengh}
    \sum_{\sigma \in \theta} p_{\sigma} = \|\theta\|, \; \P-a.s.
    \end{equation}
\end{lemma}
\begin{proof} Given a follow $\theta$, we can associate to a directed path $\sigma$ from $0$ to $\partial B_R$ the following weight 
	\begin{equation}
	p'_{\theta, \sigma} := \min_{e \in \sigma}\theta(e). 
	\end{equation}

	We will obtain our $p$-weights to satisfy Equation \eqref{eq:pweights_strengh} from the $p'$-weights in an inductive way.
	First choose a path $\sigma_0$ such that $p'_{\theta,\sigma_0} >0$, which exists due to the fact that $\|\theta \| > 0$. Now, consider the new flow 
  $$
  \theta_0 = \theta - p'_{\theta, \sigma_0} \mathbb{1}_{\{e \in \sigma_0\}}.
  $$
  If there is no other directed path from $0$ to $\partial B_R$ whose $p'$-weight under $\theta_0$ is positive, then $\| \theta_0 \| = \mathrm{div}(0) = 0$, since we have removed from $\theta$ the only path leading flow from $0$ to $\partial B_R$, and we set $p_{\sigma_0} := p'_{\theta, \sigma_0}$ . However, if there is another path $\sigma_1$ such that $p'_{\theta_0,\sigma_1} > 0$, then we set 
 $$
	p_{\sigma_0} := p'_{\theta, \sigma_0}; \quad p_{\sigma_1} := p'_{\theta_0, \sigma_1},
 $$
  and we consider a new flow
  $$
  \theta_1 := \theta - p_{\sigma_0} \mathbb{1}_{\{e \in \sigma_0\}} - p_{\sigma_1} \mathbb{1}_{\{e \in \sigma_1\}}.
  $$

  Repeating this procedure until we end up with a degenerate flow $\theta_k$, which allows us to write
  \begin{equation}\label{eq:thetadecomp}
    \theta = \theta_k + \sum_{\sigma \in \theta}p_{\sigma} \mathbb{1}_{\{e \in \sigma\}}. 
  \end{equation}
  Finally, the above identity yields
  \begin{equation}
    \begin{split}
      \|\theta \| & =  \sum_{e, e_{-} = 0}\theta(e) - \sum_{e, e_+ = 0}\theta(e) \\ 
      & = \|\theta_k \| + \sum_{e, e_{-} = 0}\sum_{\sigma \in \theta}p_{\sigma} \delta_{\{e \in \sigma\}} - \sum_{e, e_{+} = 0}\sum_{\sigma \in \theta}p_{\sigma} \delta_{\{e \in \sigma\}} \\ 
      & = \sum_{e, e_{-} = 0}\sum_{\sigma \in \theta}p_{\sigma} \delta_{\{e \in \sigma\}} = \sum_{\sigma \in \theta}p_\sigma\delta_{\{e \in \sigma\}}, 
    \end{split}
  \end{equation}
since $\|\theta_k \|=0$ and all the directed paths $\sigma$ from $0$ to $\partial B_R$ do not contain any directed edge returning to $0$ but only leaving $0$, which implies 
$$
\sum_{e, e_{+} = 0}\sum_{\sigma \in \theta}p_{\sigma} \delta_{\{e \in \sigma\}} = 0\; \text{  and } \sum_{e, e_{-} = 0}\sum_{\sigma \in \theta}p_{\sigma} \delta_{\{e \in \sigma\}} = \sum_{\sigma \in \theta}p_\sigma,
$$
which concludes the proof of the lemma.
\end{proof}
We are now able to state and prove the connection between path probabilities and flows.

\begin{lemma}[From flows to paths] \label{lemma:flowpath} For a fixed positive integer $R$, let $\theta$ be a (random) flow from $0$ to $\partial B_R$ supported on the directed version of $B_R$. Then, for any $q>0$
	\[
	\P \left( \max_{y \in \partial B_R}P_{0, \omega}\left( H_y < H_0^+ \right) \le q \right) \le \P \left( \prod_{e \in \mathcal{E}(B_R)} \omega(e)^{\theta(e)} \le q^{\|\theta\|}\right). 
      \]
\end{lemma}
\begin{proof}  We begin by noticing that we may assume $\|\theta \| >0$, $\P$-a.s. This is possible because on the environments such that $\|\theta \| = 0$, we have $q^{\|\theta\|} = 1$ and then the result trivially holds.

Since $\theta$ is a non-degenerate flow $\P$-a.s., by Lemma \ref{lemma:pweights}, $\theta$ induces a set of paths from $0$ to $\partial B_R$ with positive $p$-weights. On the other hand, given a path $\sigma$ from $0$ to $\partial B_R$ we may associate a weight to it according to the transitions probabilities on $B_R$. I.e., 
\[
\omega_{\sigma} := \prod_{e \in \sigma} \omega(e). 
\]
Observe that the following inequality holds 
\begin{equation}\label{ineq:max}
\P \left( \max_{y \in \partial B_R}P_{0, \omega}\left( H_y < H_0^+
  \right) \le q \right) \le  \P \left( \omega_{\sigma} \le q, \,
  \forall \sigma \text{ from } 0 \text{ to }\partial B_R \text{ in }\theta\right). 
\end{equation}
Now, on the event $\omega_{\sigma} \le q$ for all $\sigma \in \theta$ and by Lemma \ref{lemma:pweights}, we have
\begin{equation}
\sum_{\sigma \in \theta} p_{\sigma}\omega_{\sigma} \le q \sum_{\sigma \in \theta} p_{\sigma} = q\|\theta\|. 
\end{equation}
But, by Jensen's inequality for concave functions we obtain 
\begin{equation}\label{eq: jensen}
\begin{split}
\log q \ge \log \left( \|\theta\|^{-1} \sum_{\sigma \in \theta} p_{\sigma}\omega_{\sigma} \right) & \ge \|\theta\|^{-1}\sum_{\sigma \in \theta} p_{\sigma}\log \omega_{\sigma} = \log \left( \prod_{ \sigma \in \theta} \omega_{\sigma}^{p_{\sigma}/\|\theta\|}\right) \\
q  \ge \|\theta\|^{-1} \sum_{\sigma} p_{\sigma}\omega_{\sigma} & \ge \prod_{ \sigma \in \theta} \omega_{\sigma}^{p_{\sigma}/\|\theta\|} 
\end{split}
\end{equation}
By Equation \ref{eq:thetadecomp}, it follows that for a fixed edge $e$ we have 
\[
\sum_{\sigma, \; \sigma \ni e}p_{\sigma} \le \theta(e). 
\]
Thus, by the above inequality and the fact that $\omega(e)$ is at most $1$ for all edges, the product~$\prod_{ \sigma \in \theta} \omega_{\sigma}^{p_{\sigma}/\|\theta\|}$ can be bounded from below by 
\[
\prod_{ \sigma \in \theta} \omega_{\sigma}^{p_{\sigma}/\|\theta\|} \ge \prod_{ \substack{e \in \mathcal{E}(B_R) \\ \exists \sigma, e \in \sigma}} \omega(e)^{\|\theta\|^{-1}\sum_{\sigma, \; \sigma \ni e}p_{\sigma}}\prod_{ \substack{e \in \mathcal{E}(B_R) \\ \nexists \sigma, e \in \sigma}} \omega(e)^{\|\theta\|^{-1}\theta(e)} \ge \prod_{e \in \mathcal{E}(B_R)} \omega(e)^{\|\theta\|^{-1}\theta(e)}. 
\]
Returning to Equation \eqref{eq: jensen}, we conclude that on the event $\omega_\sigma \le q$ for all $\sigma \in \theta$, the following holds
$$
\prod_{e \in \mathcal{E}(B_R)} \omega(e)^{\|\theta\|^{-1}\theta(e)} \le q,
$$
which combined with Equation \eqref{ineq:max} leads to 
\begin{equation}\label{ineq:flowrelation}
\P \left( \max_{y \in \partial B_R}P_{o, \omega}\left( H_y < H_0^+ \right) \le q \right)  \le \P \left( \prod_{e \in \mathcal{E}(B_R)} \omega(e)^{\theta(e)} \le q^{\|\theta\|}\right), 
\end{equation}
which proves the lemma. 
\end{proof}

\medskip

\subsubsection{The exit strategy}\label{sss3} The next step towards
proof of Theorem \ref{thm:X} is to construct an exit
strategy for the random walk from the a ball $B_R$.
To construct this strategy, we will first need to construct two
auxiliary processes, which we will call
the \textit{exploration processes}, each one of which  choses a set of paths between
$0$ and the boundary $\partial B_R$.
We will then use the set of  vertices defined by the paths of both
exploration processes to generate
a subgraph $G_R$ of $B_R$ for which we can control the path probabilities. This subgraph can be seen as a simplification of $B_R$ but large enough to contain good paths from the origin $0$ to $\partial B_R$.

Now we can define the two exploration processes involved in the exit 
strategy. At any given time, the exploration process is defined as a
set of {\it activated vertices}, a set of {\it deactivated vertice}
and an integer
keeping track of the number of bifurcations done in the exploration.
Each exploration process will be denoted by $\{\explor{(i)}{t}\}_t$,
with $i=1,2$. The state space of each one is
$\mathcal P(\mathbb Z^2)^2
\times \mathbb N$, where $\mathcal P(\mathbb Z^2)$ is the power set of
$\mathbb Z^2$, so that
each at each time $t$ we have 
$$
\explor{(i)}{t} := (\mathcal{A}_t^{(i)}, \mathcal{D}^{(i)}_t,
B_t^{(i)})\in \mathcal P(\mathbb Z^2)^2\times\mathbb N,
$$
where $\mathcal{A}^{(i)}_t$ stands for the \textit{activated} vertices at time $t$ and $\mathcal{D}^{(i)}_t$ for the \textit{deactivated} ones in the $i$-th exploration process, whereas $B_t^{(i)}$ stands for the number of bifurcation rules we have used during the $i$-th exploration process. 
Each exploration process will evolve as a random subset of activated
and deactivated sites, so that at each time, new activated sites are
added which are nearest neighboring sites to the active sites, while
some old active sites become deactivated.
To define this evolution precisely  we  need
to define rules of activation of new sites which we call
 \textit{activation rules}.   Below we describe each activation rule and then we will see how the exploration processes use them. 
 evolution
Throughout all the definitions, we assume the activation rule will be
performed from a fixed vertex $x$
which is active. Moreover, when a vertex becomes
active,  a new active vertex will be
attached to it later according to a certain rule which is a function
of the environment, and which we will call an {\it activation rule} or {\it instruction} which
will
be denoted by
~$\mathcal{I}(x)$\footnote{Latter, the instructions
  on~$\mathcal{I}(x)$ will help us to decide which activation rule may
  be used on $x$.}. The activation rule or instruction that is
performed at each step will depend on the environment. This is the
list of possible activation rules from a vertex $x$, where in all of them
$i,j\in\{1,\ldots,d,-1,\ldots,-d\}$ are
the indices unit vectors $e_i, e_j, e_k$:
 
\medskip

\begin{itemize}
	\item  \underline{\textit{Forward rule to direction $j$}.}
          Activate vertex $x+e_j$ and $\mathcal{I}(x+e_j)$ becomes the instruction ``\textit{forward-$j$}''. \\
	
	\item \underline{\textit{Orthogonal rule for $i$ and $j$}.} In
          the case in which $i$ and $j$ are orthogonal directions, activate vertex $x+e_k$, where 
	\begin{equation}
	k := \arg \max \{\omega(x, x+e_i), \omega(x, x+e_j)\}. 
	\end{equation}
	In case of tie, we simply choose $k = \min\{i,j\}$.
        $\mathcal{I}(x+e_k)$ becomes the instruction ``\textit{orthogonal}-$(i,j)$''\\
	
	\item  \underline{\textit{First bifurcation rule of direction $j$}.} Activate the following vertices: $x+e_{k_1}$ and $x+e_{k_2}$, where 
	\begin{equation}
	k_1 := \arg \max_{k \neq -j} \{\omega(x, x+e_k)\}; \, k_2 := \arg \max_{k \notin \{-j, k_1\}}\{\omega(x, x+e_k)\}. 
	\end{equation}
	 $\mathcal{I}(x+e_{k_1})$ becomes the instruction
         ``\textit{forward}-$k_1$'' while $\mathcal{I}(x+e_{k_2})$
         becomes the instruction ``\textit{orthogonal-}$(j,-k_1)$''. \\
	
	\item  \underline{\textit{Second bifurcation rule of direction $j$}.} Activate the following vertices: $x+e_{k_1}$ and $x+e_{k_2}$, where 
	\begin{equation}
	k_1 := \arg \max_{k \neq -j} \{\omega(x, x+e_k)\}; \, k_2 := \arg \max_{k \notin \{-j, k_1\}}\{\omega(x, x+e_k)\}. 
	\end{equation}
$\mathcal{I}(x+e_{k_1})$  becomes the instruction
``\textit{orthogonal}-$(j,k_1)$'' while $\mathcal{I}(x+e_{k_2})$ becomes the instruction ``\textit{orthogonal-}$(j,-k_1)$''. 
\end{itemize}

\medskip

Now we define the initial conditions of both processes and then
describe how they evolve according
to the activation rules (given an environment their evolution will be
independent, an they will only differ in their initial condition):
for $\mathcal Z^{(1)}$ we set its initial condition as the one whose
only activated site is $0$, no deactivated sites and $B^{(1)}_0=0$, so that
\begin{equation}
  \nonumber
\mathcal Z^{(1)}_0=(\mathcal{A}^{(1)}_0,  \mathcal{D}^{(1)}_0, B^{(1)}_0) =
(\{0\},\emptyset ,0)
\end{equation}
For $\mathcal Z^{(2)}$, its initial condition will also be chosen as
one having only one activated site, no deactivated sites and
$B^{(2)}_0=0$.
Nevertheless, its activated site will be chosen as the nearest
neighbor of $0$ to which there is a highest probability of jumping from
$0$. To define this, let

\begin{equation}
  \nonumber
j_* := \arg \max_{k \in \{-2,-1,1,2\}} \{\omega(x, x+e_k)\}, 
\end{equation}
In case of tie, choose $j_*$ arbitrarily. We denote $0':= e_{j_*}$ and set 
\begin{equation}
  \nonumber
\mathcal Z^{(2)}_0=(\mathcal{A}^{(2)}_0,  \mathcal{D}^{(2)}_0,
B^{(2)}_0) = (\{0'\}, \emptyset, 0). 
\end{equation}
 We furthermore set $\mathcal{I}(0)$  as the instruction
 ``\textit{forward-$(-j_*)$}'' and
 $\mathcal{I}(0')$ as ``\textit{forward-$j_*$}''.

Let us now define the evolution of our processes. In the discussion
below $i=1$ or $i=2$.
Suppose that at a given time $n$ the $i$-th process is in state
$\explor{(i)}{n}$. Order the sites of $\mathbb Z^d$ according to the
lexicographic order and select $x \in \mathcal{A}^{(i)}_n$ as
the smallest site. We then execute on $x$ the update rule described below. \\

\noindent \underline{Update rule:} 
\begin{itemize}
	\item \underline{Case 1: $\mathcal{I}(x)$ is the instruction ``\textit{forward}-$j$''}.  If 
	\[
	j \in  \arg \max_{k \neq -j} \{\omega(x, x+e_k)\}, 
	\]
	we activate site $x+e_j$ if
        $x+e_j\notin\partial B_R$, so that we set 
	$$
	\mathcal{A}^{(i)}_{n+1} = 
	\{(x+e_{j})\mathbb{1}_{\{x+e_{j} \notin \partial B_R\}}\} \cup \mathcal{A}^{(i)}_n \backslash \{x\}, 
	$$
	where the notation 
	$$
	(x+e_{j})\mathbb{1}_{\{x+e_{j} \notin \partial B_R\}}
	$$
	means  that we add the element  $x+e_{j}$ only if the condition under the indicator function is satisfied. We also put 
	$$\mathcal{D}^{(i)}_{n+1} = \{x,
        (x+e_{j})\mathbb{1}_{\{x+e_{j} \in \partial B_R\}} \} \cup
        \mathcal{D}^{(i)}_n. $$
        We will say that in this case a new site was activated in the
        {\it forward direction}.

	Otherwise, if

        	\[
	j \notin  \arg \max_{k \neq -j} \{\omega(x, x+e_k)\}, 
      \] 
      a bifurcation will be produced, so  we either perform  the first
      bifurcation rule if $B_n^{(i)} \mod{2} = 0$ and
      otherwise the second one. Then we set $B_{n+1}^{(i)} := B_{n}^{(i)} +1$ and 
	$$\mathcal{A}^{(i)}_{n+1} = \{(x+e_{k_1})\mathbb{1}_{\{x+e_{k_1} \notin \partial B_R\}}, (x+e_{k_2})\mathbb{1}_{\{x+e_{k_2} \notin \partial B_R\}} \} \cup \mathcal{A}_n \backslash \{x\}.$$
and 
	$$\mathcal{D}^{(i)}_{n+1} = \{x, (x+e_{k_1})\mathbb{1}_{\{x+e_{k_1} \in \partial B_R\}}, (x+e_{k_2})\mathbb{1}_{\{x+e_{k_2} \in \partial B_R\}} \} \cup \mathcal{D}^{(i)}_n.$$
We will say in this case that a {\it bifurcation} was produced.

	In summary, if jumping to $x+e_j$ has the largest probability among all directions but $-j$, we take a step to direction $j$. Otherwise, we bifurcate activating the two vertices with highest transition probabilities among all directions, expect $-j$. \\
	
	\item \underline{Case 2: $\mathcal{I}(x)$ is the instruction
            ``\textit{orthogonal}-$(i,j)$''}.
          In this case, we apply the orthogonal rule and make the
          update to time $n+1$,
          
	$$
\mathcal{A}^{(i)}_{n+1} = 
\{(x+e_{k_*})\mathbb{1}_{\{x+e_{k_*} \notin \partial B_R\}}\} \cup \mathcal{A}^{(i)}_n \backslash \{x\}
$$
	and 
$$\mathcal{D}^{(i)}_{n+1} = \{x, (x+e_{k_*})\mathbb{1}_{\{x+e_{k_*} \in \partial B_R\}} \} \cup \mathcal{D}^{(i)}_n,$$
	where ${k_*}$ is the direction given by the \textit{orthogonal-(i,j)} rule.		
	
\end{itemize}

\medskip
We now run independently both exploration processes $\{\explor{(1)}{t}\}_t$ and~$\{\explor{(2)}{t}\}_t$ until both have stopped, which occurs when their set of activated vertices is empty. Both processes stop with probability one since at each step we increase the distance from $0$ (or $0'$) considering paths using activated or deactivated vertices. 

Let $\tau_{i}$ denote the time $\{\explor{(i)}{t}\}_t$ stops. We let $\comp{(i)}{R}$ be the subgraph of $B_R$ whose vertex set is 
\begin{equation}
V(\comp{(i)}{R}) = \mathcal{D}^{(i)}_{\tau_i}. 
\end{equation}
We then construct the subgraph $G_R$ generated by the whole strategy: 
\begin{equation}
V(G_R) := \mathcal{D}^{(1)}_{\tau_1} \cup \mathcal{D}^{(2)}_{\tau_2}. 
\end{equation}

We end this section dedicating a few lines to give some examples of
the kind of graphs the exit strategy may generate. All the figures
below represents $\comp{(1)}{R}$.
In three figures below, the strongest arrow
means this was the direction selected by the update rule, whereas the
light gray arrows represents the other directions the rule had to
check. In case of the first picture in Figure 2.2,
$e_{j_*} = e_{-1}$ and the $\{\explor{(1)}{t}\}_t$ successfully
applied the \textit{forward rule to direction} $1$, $R$ times in a roll. 
\begin{figure}\label{fig:comp1}
	\begin{center}
\scalebox{1}{
	\begin{tabular}{cc}
		
		\includegraphics[scale=0.8]{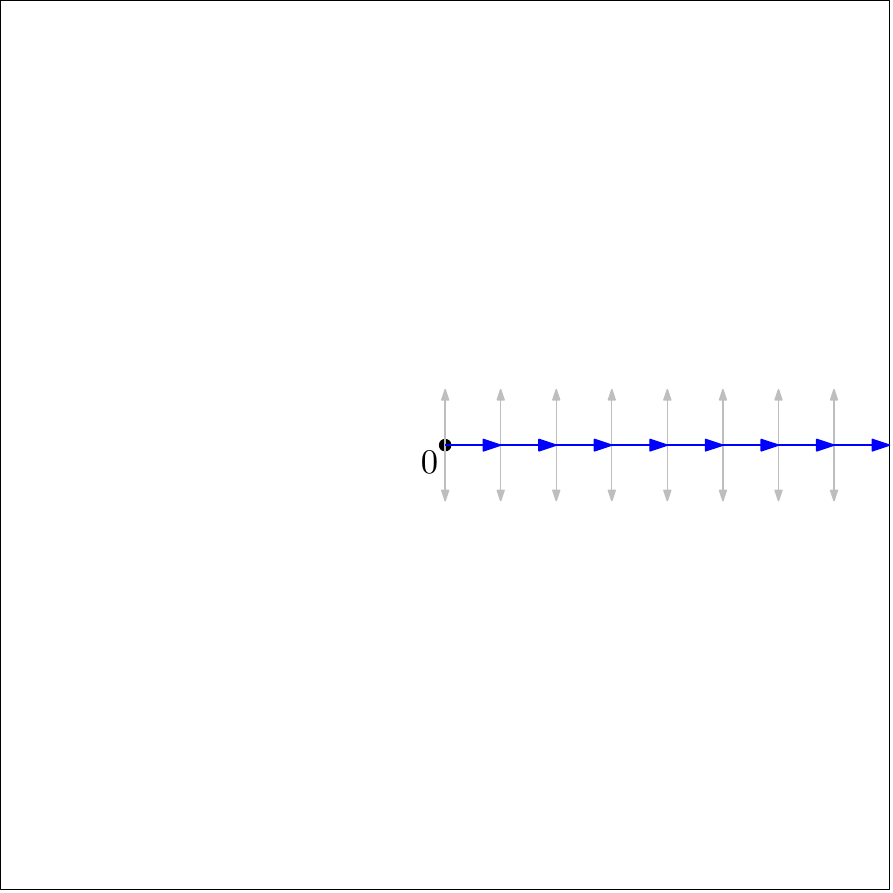}
		&
		\includegraphics[scale=0.8]{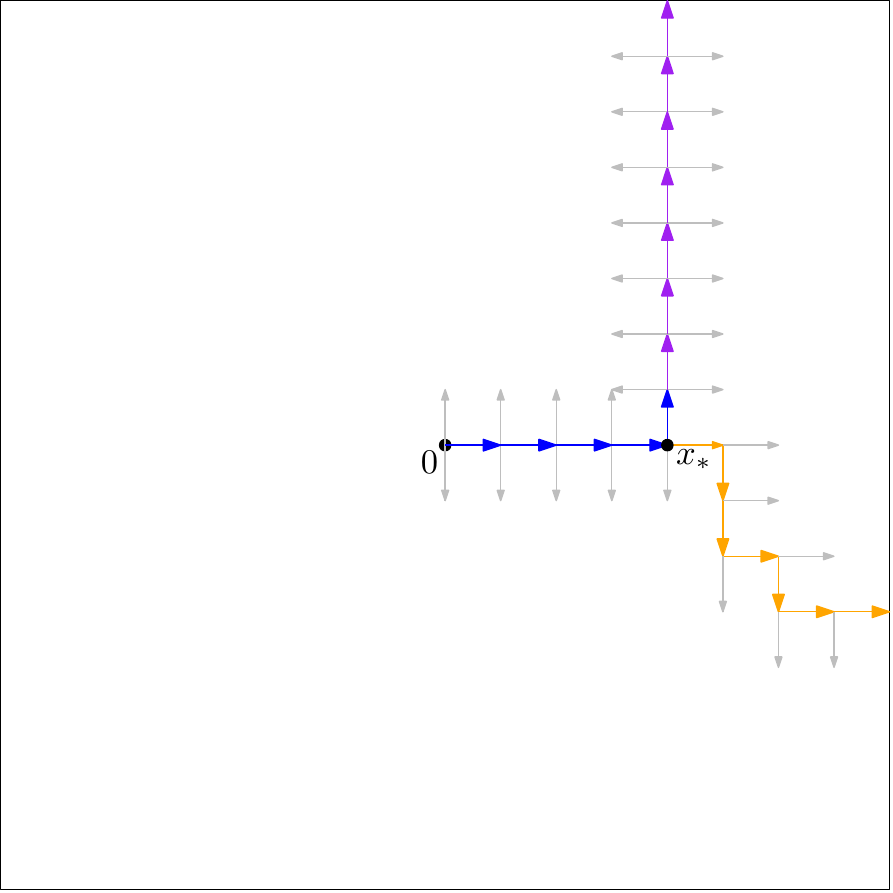}
		\\
		\multicolumn{1}{c}{ } & \multicolumn{1}{c}{} \\
		\multicolumn{1}{c}{\small Forward rule to direction $1$ applied $R$ times}
		&
		\multicolumn{1}{c}{\small Bifurcation rule of direction $1$ applied to $x_*$}
	\end{tabular}
      }
      \caption{\ }
\end{center}
\end{figure}
Whereas, in the second picture of Figure 2.2, after
applying the \textit{forward rule} a few times, the first bifurcation rule  is applied on $x_*$. Thus we \textit{activate} two new vertices: one with an \textit{orthogonal} instruction, which is followed until we reach the boundary, and another vertex with a "forward-$2$" instruction. Then, the process successfully apply the activation rule \textit{forward rule to direction $2$}, generating a up-path from $x_*$ to the boundary of the box. 
Finally, in Figure 2.3 we have an example where the process bifurcates twice. Notice that in each component the process bifurcates at most two times, since in the second bifurcation, the activated vertices receives orthogonal instructions, thus from them we keep applying the \textit{orthogonal} rule. 
\begin{figure}\label{fig:comp22}
	\centering 
	\includegraphics[scale=0.8]{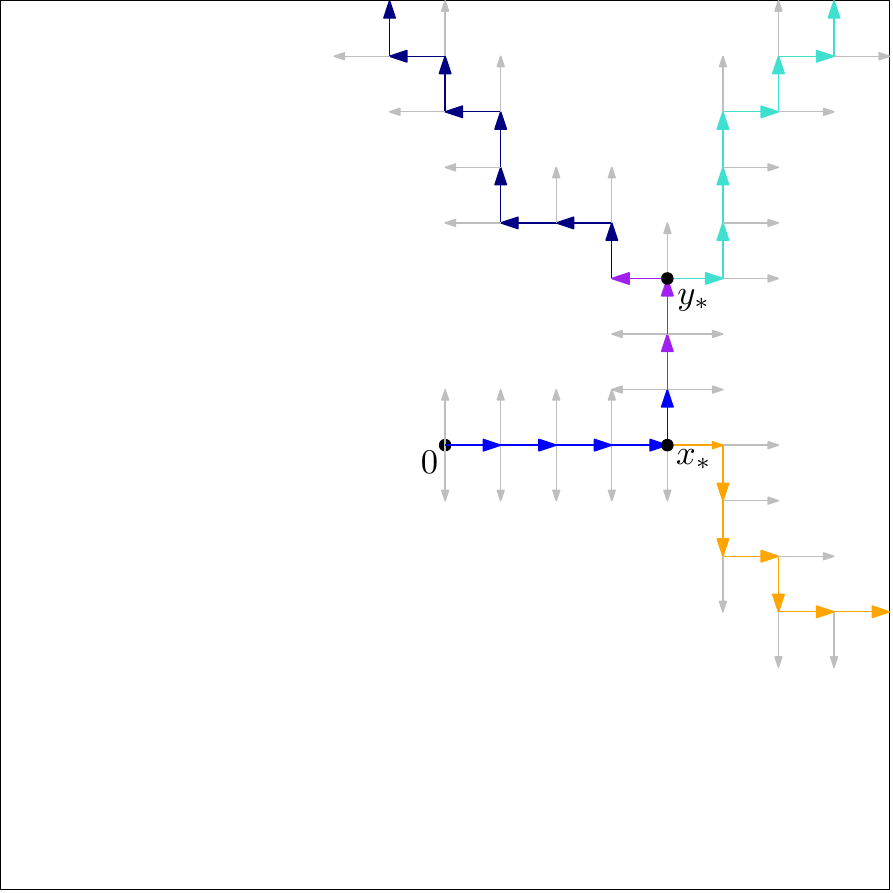}
	\caption{Two bifurcations $x_*$ and $y_*$}
      \end{figure}

      \medskip

\subsubsection{Constructing random capacities on $B_R$}\label{sec:network}
As said before we want to guarantee the existence of a (random) flow
by applying the the Max-flow Min-cut Theorem stated at Theorem
\ref{th:maxflow}. Thus we need to construct a network in $B_R$. In order to do that, first we see $B_R$ as a directed graph. That is, each edge of $B_R$ appears twice (one for each direction). Whereas the edges of $G_R$ appear only once and in the direction they have been revealed by the exploration process. That is, if from $x_n^{(i)}$ we have activated $x_n^{(i)} +e_j$, then only $(x_n^{(i)}, x_n^{(i)}+e_j)$ belongs to the directed version of $G_R$. 

Next we must give the directed edges of $B_R$ a capacity. This
capacity function $c$ depends on the environment, since it will depend on the random graph $G_R$. However, to keep the notation compact, we will omit its dependence on the environment. Moreover
it will be supported on the directed edges of $G_R$, that is 
\begin{equation}
c\restriction_{\mathcal{E}(B_R)\setminus\mathcal{E}(G_R)} = 0. 
\end{equation}

Now, let us describe how we construct the capacity function
$c$. Consider an edge $(x,x+e_j) \in \mathcal{E}(G_R)$ such that $x$
belongs to a
single component $\mathcal{C}_R^{(i)}$ and such that no  bifurcation
rule has been performed on it, we assign  the  capacity 
$$
c((x,x+e_j)) = \begin{cases}
			\beta_j, & \text{ if }\mathcal{I}(x) = \text{ forward-}j \\ \\
			\alpha_{ij}, & \text{ if }\mathcal{I}(x) = \text{ orthogonal-}(i,j) 
		\end{cases}
                $$
  where the  exponents $\beta_j$ and $\alpha_{ij}$ and 
the transition probabilities $Q_{\vdash}$ and $Q_{\llcorner}$ have
been defined  in Section \ref{sec:notation}.
Notice that if $(x,x+e_j) \in \mathcal{E}(G_R)$ and $\mathcal{I}(x) = \text{ forward-}j$ it means that 
$$
\omega(x, x+e_j) = \max_{k \neq -j}\omega(x, x+e_k), 
$$
which in turn implies that 

\begin{equation}
  \label{intint}
\int_{\omega(x, x+e_j) = \max_{k \neq -j}\omega(x, x+e_k)} \omega(x,x+e_j)^{-\beta_j}\mathrm{d}\P(\omega) < \infty.
\end{equation}
In this case we say
$\omega(x, x+e_j)$ has a singularity of at
least $\beta_j$ on the event $\{\omega(x, x+e_j) = \max_{k \neq 
  -j}\omega(x, x+e_k)\}$.
In what follows we will say that $\omega(x,x+e)$ has a singularity
of at least $\alpha$ on the event $A$ if the integral as
(\ref{intint}), where $\omega(x,x+e_j)$ is replaced by
$\omega(x,x+e)$, $\beta_j$ by $\alpha$ and $\{\omega(x, x+e_j) = \max_{k \neq 
  -j}\omega(x, x+e_k)\}$ by $A$, is finite. 
The picture
below illustrates this for the case  $j=1$. 
\begin{figure}[h]
	\centering 
	\includegraphics[width=0.3\linewidth]{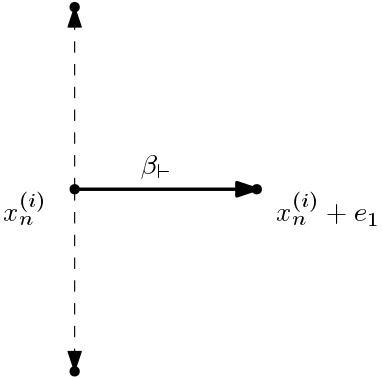}
	\caption {Assigning capacity $\beta_{\vdash} = \beta_1$ to the edge $(x, x +e_1) \in \mathcal{E}(\mathcal{C}^{(i)}_R)$}
	\label{fig:forward-1}
      \end{figure}
The same argument works when $\mathcal{I}(x) = \text{ orthogonal-}(i,j)$. In this case 
$$
\omega(x, x+e_j) = \max_{k \in \{i,j\}}\omega(x, x+e_k) 
$$
and then $\omega(x, x+e_j)$ has a singularity at least $\alpha_{i,j}$
on the event $\{\omega(x, x+e_j)$  $= \max_{k \in \{i,j\}}\omega(x, x+e_k)\}$.
Notice that by construction, when $x$ is a vertex in which a
bifurcation rule has been performed the instruction attached to $x$
must be \textit{forward-}$\ell$ for some direction $\ell$. Then, still
considering the case
in which the edge belongs to a single component, we assign the following capacity to $(x,x+e_j)$
\begin{equation}
c((x,x+e_j)) = \begin{cases}
\beta_j, & \text{ if }\omega(x, x+e_j) = \max_{k \neq -\ell}\omega(x, x+e_{k}); \\ \\
\alpha_{\ell j}, & \text{ otherwise. }
\end{cases}
\end{equation}
The picture below illustrates the two cases at once. When $\ell = 1$, but the largest transition probability among directions $\{-1,1,2\}$ is at direction $2$, we assign capacity $\beta_2 = \beta_{\perp}$. And for the edge which has the largest probability transition among $\ell$ and $-2$ we assign capacity $\alpha_{1,(-2)} = \alpha_{\stackrel{}{\ulcorner}}$. 
\begin{figure}[h]
	\centering 
	\includegraphics[width=0.3\linewidth]{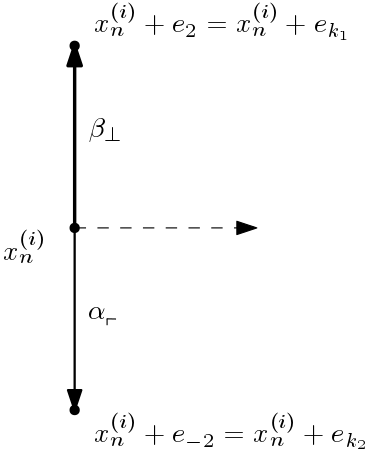}
	\caption{}
	\label{fig:bifurcation}
\end{figure}
When an edge belongs to both components,  $\mathcal{C}_R^{(1)}$ and $\mathcal{C}_R^{(2)}$, we assign the rule of assignment which gives the largest capacity. That is, for an edge that belongs to both components, the above procedure applied for each component gives us two possible capacities for the edge, then we choose the largest one.

\medskip

\subsubsection{Constructing a random flow from $0-0'$ to $\partial B_R$}
 Our  objective in this section is to construct a flow on the network
  $B_R$ which has the capacity constructed in the previous section. This flow will have the pair of vertices $\{0,0'\}$ as source and $\partial B_R$ as sink. So, in order to simplify our notation, we will write $0-0'$ instead of $\{0,0'\}$. 
With that in mind, formally, the objective of this section is to prove the following theorem.

\medskip

\begin{theorem}[Random flow from $0-0'$ to $\partial
  B_R$]\label{thm:randomflow} Let $a>0$. Consider an i.i.d. random environment
  satisfying condition $(X)_a$. Then, for any radius $R$ there exist a
  $\varepsilon>0$ and a random flow $\theta$ from $0-0'$ to $\partial
  B_R$ such that

  \begin{enumerate}
		\item $\|\theta\| \ge a+\varepsilon$, $\P$-almost surely; 
		
		\item $\theta(e) \le c(e)$  for all directed edges in $\mathcal{E}(B_R)$, $\P$-almost surely. Here $c$ is the capacity function constructed in Section \ref{sec:network}
	\end{enumerate}
	
\end{theorem}
\begin{proof}
	Given the relations \eqref{eq:edgerel}-\eqref{eq:squarerel} on
        condition $(X)_a$, there exist small enough $\varepsilon$ so
        that all relations are  still satisfied with $a$ replaced by $a+\varepsilon$. 
	
	Now, for a fixed environment $\omega$ we run our exploration process which gives us $G_R$ and the correspondent capacity $c$. By the Max-flow Min-cut Theorem (Theorem \ref{th:maxflow}) in order to prove the existence of a flow from the source $\{0, 0'\}$ to the sink $\partial B_R$ satisfying (1) and (2) we have to prove that 
	\begin{equation}\label{eq:capcut}
		\min \{c(\Pi)\, :\, \Pi \textit{ is a cutset }\} \ge a+\varepsilon, 
	\end{equation}
	where the capacity of a set of edges is just 
	$$
	c(\Pi) = \sum_{e \in \Pi}c(e). 
	$$
	But observe that every cut set $\Pi$ must contain edges of
        $G_R$, otherwise we would have a path from $0$ (or $0'$) to
        $\partial B_R$ with no edge in $\Pi$. Moreover, since we want to minimize the capacity of cut sets, we must separate $0-0'$ from $\partial B_R$ using edges with the smallest capacity and using the smallest number of edges possible. 
	
	Moreover, by construction, all the edges whose capacity is
        positive must have a capacity given by an exponent of either
        $\beta$-type or $\alpha$-type. Also by the construction of the
        exit strategy the capacity of $\Pi$ is bounded from below by
        some combination of the exponents covered by one of the
        inequalities \eqref{eq:edgerel}-\eqref{eq:squarerel}. Thus, by the Max-flow Min-cut theorem follows our result. 

\end{proof}

\medskip

\subsubsection{Final step of the proof.
  }
Before we start the proof of the two main results of this section, we
will need an additional lemma.

\medskip

\begin{lemma}\label{lemma:expecfinite} Let $a>0$. Assume that condition $(X)_a$ holds, then
  $$
	\mathbb{E}\left[ \prod_{e \in \mathcal{E}(B_R)}\omega(e)^{-c(e)}\right] < \infty. 
	$$
where $c$ is the random capacity constructed at Section \ref{sec:network}.
\end{lemma}
  
\begin{proof}
	In order to prove the above finiteness, we split the expected
        value into the possible realizations for the exploration
        processes $\{\explor{(1)}{t}\}_t$ and $\{\explor{(2)}{t}\}_t$. But before we start the proof, it might be instructive to examine in more details the evolution of the exploration process. One important feature of the exploration process is that it evolves by examining the transition probabilities of one given vertex at each step. This implies that an intersection of events like the ones below
        \begin{equation}\label{eq:event_int}
           \Lbrace \arg \max_k \omega(0,e_k) = j, \arg \max_{k\neq j} \omega(0,e_k) = -j \Rbrace \cap \Lbrace \arg \max_{k\neq j} \omega(e_{-j},e_k) = -j \Rbrace \cap \dots,
        \end{equation}
        completely determines the exploration process. In other words, if we know the transition probabilities for some subset of vertices of $B_R$ this is enough to determine the evolution of the exploration process. Thus, we let
        $$
        \Lbrace \{\explor{(i)}{t}\}_t = Z_i \Rbrace, 
        $$
        be an arbitrary event of the type we exemplified in \eqref{eq:event_int}, which completes determines $\{\explor{(i)}{t}\}_t$. And to simplify our writing, we write
	    \begin{equation}\label{eq:z_trajectories}
  \Lbrace \mathcal{Z} = Z\Rbrace := \Lbrace \{\explor{(1)}{t}\}_t = Z_1, \{\explor{(i)}{t}\}_t = Z_2\Rbrace, 
      \end{equation}
   which determines the evolution of both exploration processes.
       To keep the notation compact on the event $\{\mathcal{Z} = Z\}$,
        we write $G_R=G$. 
        
  Since the random capacity $c$ is supported on the the random graph $G_R$, we have that 
        $$
        \prod_{e \in \mathcal{E}(B_R)}\omega(e)^{-c(e)}\mathbb{1}\{\mathcal{Z} = Z\} = \prod_{e \in \mathcal{E}(G)}\omega(e)^{-c(e)}\mathbb{1}\{\mathcal{Z} = Z\} \quad \P-\text{a.s.,}
        $$
      which implies that it is enough to prove
	\begin{equation}\label{eq:prod}
    \mathbb{E}\left[ \prod_{e \in \mathcal{E}(G)}\omega(e)^{-c(e)}\mathbb{1}\{\mathcal{Z} = Z\}\right] < \infty.
	\end{equation}
and then use the fact the $B_R$ has finite volume. In order to prove the above, for a fixed edge $e = (x, x+e_j) \in \mathcal{E}(G)$, we consider
whether there is another edge $e' = (x, x+e_i) \in \mathcal{E}(G)$ or
not. We want to use the i.i.d. nature of the environment to write the
above expectation  as a product of other expectations. However, since
we may have distinct edges which are adjacent to the same vertex, we
will have to arrange the product grouping all the terms coming from
the same vertex, then we can use independence between vertices and
finally conditions \eqref{x:eq1} and \eqref{x:eq2} to guarantee that each expectation in the product is finite. 

Since we can write $\mathbb{1}\{\{\explor{(i)}{t}\}_t = Z_i\}$ as a product of independent indicators indexed by subset of vertices of $B_R$, given a vertex $x$ in this index set, we will write $\mathbb{1}\{\mathcal{Z}_t^{(i)}(x) = Z_i(x)\}$ to denote the $x$-th indicator in this product.  We then continue the proof separating it in different cases
according to the position of $x$ in $G$. \\
	
	\noindent \underline{Case 1. Vertex $x$ has only one neighbor in $G$:} \\
		In this case, $x$ is a vertex of $B_R$ such that for
                some $j$, $(x, x+e_j)$ belongs to $\mathcal{E}(G)$ and
                $(x,x+e_i) \notin \mathcal{E}(G)$ for all $i\neq
                j$. Combining this with the fact that $c((x,x+e_j))$ is not random on the event $\{\mathcal{Z} = Z\}$, it follows that the term $\omega(x,x+e_j)^{-c((x,x+e_j))}$ is independent of all other random variables of the form~$\omega(e)^{-c(e)}$ inside the expected value, except possibly from $\mathbb{1}\{\mathcal{Z} = Z\}$. 
	
	Moreover, recall that the random capacity $c((x, x+e_j))$ is a number chosen according to the instruction $\mathcal{I}(x)$ and the distribution of $\omega(x, \cdot)$ and in a way that due to \eqref{x:eq1} and \eqref{x:eq2} it follows 
	\begin{equation}\label{eq:case1}
	\mathbb{E}\left[ \omega(x,x+e_j)^{-c((x,x+e_j))}\mathbb{1}\{\mathcal{Z}_t^{(i)}(x) = Z_i(x)\} \right] < \infty. 
	\end{equation}
	
	\bigskip

	\noindent \underline{Case 2. Vertex $x$ has more than one neighbor in $G$:} \\
	
	We first note that by construction of the exit strategy at Section \ref{sss3} there are at most two neighbors of $x$ such that $(x,x+e_i) \in \mathcal{E}(G_R)$, since all the instructions for the grow of the exploration processes do not backtrack. 
	We split this case into two subcases:\\
	
	\underline{Case 2.1. Vertex $x$ belongs to only one component $\mathcal{C}^{(i)}_R$:} In this case, we have performed a bifurcation rule on $x$. It means on the event $\{\mathcal{Z} = Z\}$ the instruction $\mathcal{I}(x)$ is ``forward-$\ell$'' to some direction $\ell$ but the largest transition among all directions different from $-\ell$ is not at direction $\ell$, this forces the process to bifurcate. And then to one edge we assign a $\alpha$-type capacity and to the other one a $\beta$-type capacity is assigned (see Figure \ref{fig:bifurcation}). Thus, we have a contribution of the form 
	\begin{equation}\label{eq:case2}
	\E\left[\omega(x,x+e_i)^{-\alpha_{i\ell}}\omega(x,x+e_j)^{-\beta_{j}}\mathbb{1}\{\arg\max_{k\neq -\ell}\omega(x,x+e_k) = j\} \right]<\infty, 
	\end{equation}
	which is finite due to \eqref{x:eq2}. \\
	
	\underline{Case 2.2. Vertex $x$ belongs to $\mathcal{C}^{(1)}_R \cap \mathcal{C}^{(2)}_R$:}
	We consider all the possible cases for $\mathcal{I}^{(1)}(x)$ and $\mathcal{I}^{(2)}(x)$ that can occur simultaneously. \\

	\underline{Case 2.2.1. $\mathcal{I}^{(1)}(x) =$``forward-j'' and $\mathcal{I}^{(2)}(x) =$ ``orthogonal-$(i,j)$'' } We first notice that if $\arg \max_{k\neq -j} \omega(x,x+e_k)= j$, then both exploration processes will activate the same neighbor and $x$ will have only one neighbor in $G$ and we already covered this case. 
	
	Thus we have to assume that $\arg \max_{k\neq -j}\omega(x,x+e_k) \neq j$. In this situation, $\{\explor{(1)}{t}\}_t$ performs a bifurcation rule at $x$, activating two neighbors of $x$, one of these neighbors is the same neighbor activated by $\{\explor{(2)}{t}\}_t$. So, this case resumes to Case 2.1. \\
	
	\underline{Case 2.2.2. $\mathcal{I}^{(1)}(x) = \mathcal{I}^{(2)}(x) =$``orthogonal-$(i,j)$'' } In this case both exploration processes activate the same vertex. Thus, even though, an intersection has occurred, $x$ has only one neighbor in $G$ and this situation is covered at Case 1. \\
	
	\underline{Case 2.2.3. $\mathcal{I}^{(1)}(x) =$``orthogonal-$(i,j)$'' and $\mathcal{I}^{(2)}(x) =$ ``orthogonal-$(-i,j)$'' } Observe that if $\arg \max_{k\neq -j}\omega(x,x+e_k) = j$ there is nothing to do by same reasoning used in the previous case. 
	
	We may assume w.l.o.g $\arg \max_{k\neq -j} \omega(x,x+e_k)= i$ and that $\arg \max_{k= \{-i,j\}} \omega(x,x+e_k)= -i$.  In this case we have the following contributions in \eqref{eq:prod}
	\begin{equation}\label{eq:case3}
	\omega(x,x+e_i)^{-\alpha_{ij}}\omega(x, x+e_{-i})^{-\alpha_{(-i)j}}\mathbb{1}\{\arg\max_{k\neq -j}\omega(x,x+e_k) = i\}. 
	\end{equation}
	Since $\alpha_{ij} \le \beta_i$, by \eqref{x:eq2} it follows that 
	$$
	\E \left[ \omega(x,x+e_i)^{-\alpha_{ij}}\omega(x, x+e_{-i})^{-\alpha_{(-i)j}}\mathbb{1}\{\arg\max_{k\neq -j}\omega(x,x+e_k) = i\} \right] < \infty. 
	$$	
	Notice that by the construction of the exit strategy cases 2.2.1, 2.2.2 and 2.2.3 cover all the possible ways of having an intersection between both exploration processes. 
	
	To conclude the proof, we first point out that seeing $\mathbb{1}\{\mathcal{Z} = Z\}$ as a product of indicators indexed by the vertices activated by each process $\{\explor{(i)}{t}\}_t$ and recalling that the event $\{\mathcal{Z} = Z\}$ completely determines the capacity and the distribution of each $\omega(x, \cdot)$ for $x$ an activated vertex. We can write $\prod_{x \in V(G)}\omega(x,x+e)^{-c((x,x+e))}\mathbb{1}\{\mathcal{Z} = Z\}$ as a product of random variables of the form given at \eqref{eq:case1}, \eqref{eq:case2} and/or \eqref{eq:case3}. This proves that for each realization of the two exploration processes 
	\begin{equation*}\mathbb{E}\left[ \prod_{e \in \mathcal{E}(G)}\omega(e)^{-c(e)}\mathbb{1}\{\mathcal{Z} = Z\}\right] < \infty, 
	\end{equation*}
	which is enough to conclude the proof since $B_R$ has finite volume and consequently there are finitely many events of the form $\{\mathcal{Z} = Z\}$ to consider.
\end{proof}
Now we have all the results needed for the main proof of this section. 
\begin{proof}[Proof of Theorems \ref{thm:X} and \ref{thm:cltX2}]
We apply our general criteria, that is, Theorems \ref{thm:BR} and \ref{thm:cltBR}. Thus, our new results will be proven if we prove that condition $(X)_a$ implies condition~$\BR{a}{\eta_*}$. Then, we have ballistic behavior under $(X)_1$ and CLT-like result under $(X)_2$. 
	
In order to prove condition $\BR{a}{\eta_*}$ holds under $(X)_a$, let $R$ be a fixed positive integer. And notice that a combination of Lemma~\ref{lemma:flowpath}, Theorem \ref{thm:randomflow}, Markov inequality and Lemma~\ref{lemma:expecfinite} implies the existence of a constant $C$ depending on $R$ such that 
\begin{equation}\label{eq:exitbr}
	\begin{split}
		\P \left( \max_{y \in \partial B_R}P_{0, \omega}\left[ H_y < H_0^+ \right] \le u^{-1} \right) &\le \P \left( \prod_{e \in \mathcal{E}(B_R)} \omega(e)^{\theta(\omega,e)} \le u^{-\|\theta\|}\right) \\
		& \le \P \left( \prod_{e \in \mathcal{E}(B_R)} \omega(e)^{c(e)} \le u^{-(a+\varepsilon)}\right) \\
		& \le u^{-(a+\varepsilon)}\mathbb{E}\left[ \prod_{e \in \mathcal{E}(B_R)}\omega(e)^{-c(e)}\right] \\
		& \le Cu^{-(a+\varepsilon)}
	\end{split}	
\end{equation}
in the second inequality we have used the properties of $\theta$ given by Theorem \ref{thm:randomflow}. Moreover, recall that the random variable $N_{B_R}(x)$ which counts the number of visits to $x$ before leaving the ball $B_R$ and let $A_n(R)$ be the following event 
\begin{equation}
A_n(R) := \Lbrace \max_{y \in \partial B_R} P_{0,w}\left[H_y < H_0^+ \right] \le n^{-\frac{a+\delta}{a+\varepsilon}} \Rbrace, 
\end{equation}
for some $\delta <\varepsilon$. Also recall that under $P_{0,\omega}$ we have 
$$
N_{B_R}(0) \sim \text{Geo}\left(P_{0,\omega}\left[ H_{\partial B_{R+1}}<H_0^+\right]\right) 
$$
and that 
\begin{equation}
P_{0,\omega}\left[ H_{\partial B_{R+1}}<H_0^+\right] \ge \max_{y \in \partial B_{R+1}} P_{0,w}\left[H_y < H_0^+ \right]. 
\end{equation}
With the above in mind, we have 
\begin{equation}
\begin{split}
P_0\left(  N_{B_R}(0) > n \right) &\le  \E\left[ P_{o,\omega}\left(  N_{B_R}(0) > n \right); A^c_n(R+1) \right] + \frac{C}{n^{a+\delta}} \\
&= \E\left[ \left(1-P_{0,\omega}\left[ H_{\partial B_{R+1}}<H_0^+\right]\right)^n; A^c_n(R+1) \right] + \frac{C}{n^{a+\delta}} \\
& \le \left(1- n^{-\frac{a+\delta}{a+\varepsilon}}\right)^n + \frac{C}{n^{a+\delta}} \\
& \le \exp\Lbrace - n^{(\varepsilon - \delta)/(a+\varepsilon)}\Rbrace + \frac{C}{n^{a+\delta}}\le \frac{C+1}{n^{a+\delta}}. 
\end{split}
\end{equation}
Using the i.i.d nature of the environment and the above, we conclude that for any fixed~$R$, vertex $x\in B_R$ and $\delta'<\delta$
\begin{equation}
E_0\left(N_{B_R}^{a+\delta'}(x)\right) <\infty. 
\end{equation} 
Finally, using $T_{B_R} = \sum_{x \in B_R} N_{B_R}(x)$ and the fact that $B_R$ has finite volume we conclude that 
$$
E_0\left(T_{B_R}^{a+\delta'}\right) <\infty, 
$$
for any $R$. Thus, under conditions $(E)_0$ and $(X)_1$ we have condition $\BR{1}{\eta_*}$. Whereas, under $(E)_0$ and $(X)_2$, condition $\BR{2}{\eta_*}$ is satisfied. This is enough to prove Theorems \ref{thm:X} and \ref{thm:cltX2}. 

\end{proof}

\section{Condition $(B)_1^{\eta_*}$: sharpness and comparison with previous condition}\label{sec:props}
In this Section we formalize the discussion made at the Introduction about optimality of condition $(B)_1^{\eta_*}$ for ballisticity. More specifically we prove Proposition \ref{prop:zerospeed}, which states that a transient walk such that $E_0[T_{B_R}] = \infty$ for some $R$ has zero speed. Recalling our discussion about condition $(B)^{\eta_*}_1$, which becomes $E_0[T_{B_R}^{1+\varepsilon}] <\infty$ for small $\varepsilon$ and large $R$, we see that $E_0[T_{B_R}] = \infty$ is essentially the complement of $(B)_1^{\eta_*}$. In this direction, Proposition \ref{prop:zerospeed} and Theorem \ref{thm:BR} implies sharpness of condition  $(B)_1^{\eta_*}$.

In this Section we also prove that condition $(B)_1^{\eta_*}$ is implied by condition $(K)_1$, proposed in \cite{FK16}, which is the most general condition prior to this work. We formalize this implication in the following
\begin{proposition}\label{prop.KBR} Consider a RWRE in an environment satisfying condition~$(K)_{a}$ in \cite{FK16}. Then, for any $b>0$, $\BR{a}{b}$ is also satisfied.
\end{proposition}
\begin{proof}By Corollary 5.1 of \cite{FK16}, condition $(K)_{a}$ implies the existence of a positive $\varepsilon$ with the following property:  for all $\delta'$ there exists $\delta$ such that
	\begin{equation}
	\P\left( \max_{y \in \partial B_{\delta\log u}} P_{0,w}\left[H_y < H_x^+ \right] \le u^{-\frac{a+2\delta'}{a+\varepsilon}}\right) \le \frac{1}{u^{a+\delta'}}.
	\end{equation}
	Now, choose $c=\varepsilon/4$, $\delta' = \varepsilon/3$, fix $R$ larger than $ a(a+c)/bc -2$ and take $u$ large enough so $\delta \log u > 2R$ and let $A_n$ be the following event
	\[
	A_n := \Lbrace \max_{y \in \partial B_{\delta\log n}} P_{x,w}\left[H_y < H_x^+ \right] \le n^{-\frac{a+2\delta'}{a+\varepsilon}} \Rbrace.
	\]
	Then, recalling that $T_{B_R}$ may be written as
	$$
	T_{B_R} = \sum_{x \in B_R} N_{B_R}(x),
	$$
	where $N_{B_R}(x)$ stands for the number of visits to $x$ before exit $B_R$, and writing 
  $$
  \tilde{Q}_x^{B_R} := P_{x, \omega}\left[ T_{B_R} < H_x^+\right]
  $$
  we have that 
	\begin{equation}\label{ineq:boundnbrx}
	\begin{split}
	P_0\left(  N_{B_R}(x) > n \right) &\le  \E\left[ P_{0,\omega}\left(  N_{B_R}(x) > n \right); A^c_n \right] + \frac{1}{n^{a+\delta'}} \\
	&= \E\left[ P_{0,\omega}\left( H_x < T_{B_R} \right)\left(1-\tilde{Q}_x^{B_R}\right)^n; A^c_n \right] + \frac{1}{n^{a+\delta'}} \\
	& \le \left(1- n^{-\frac{a+2\delta'}{a+\varepsilon}}\right)^n + \frac{1}{n^{a+\delta'}} \\
	& \le \exp\Lbrace - n^{\varepsilon/(a+\varepsilon)}\Rbrace + \frac{1}{n^{a+\delta'}}\le \frac{2}{n^{a+\varepsilon/3}},
	\end{split}
	\end{equation} 
	for large enough $n$. Since 
	$$\{T_{B_R} > n \} \subset \{ \exists x \in B_R, \; N_{B_R}(x) > n/(2R)^d \},
	$$
	using estimate \eqref{ineq:boundnbrx} and union bound (which is possible since $R$ is fixed) we may conclude that $E_0T_{B_R}^{a+c} < \infty$.
\end{proof}
Now we prove Proposition \ref{prop:zerospeed}
\begin{proof}[Proof Proposition \ref{prop:zerospeed}] Since $T_{B_R}$ is increasing in $R$,  we may assume $R$ is such that~$E_0 T_{B_{R-1}} < \infty$ whereas~$E_0 T_{B_R} =\infty$ (where  $B_0$ denotes the singleton $\{0\}$). Moreover,  there exist a point $x_0 \in \partial B_{R-1}$ such that $E_x T_{B_R} = \infty$. Otherwise,  by the Strong Markov Property
	\begin{equation}
	\begin{split}
	E_0 T_{B_R} & = E_0 \left( T_{B_{R-1}} + T_{B_R} - T_{B_{R-1}}\right) \\
	& = E_0 T_{B_{R-1}} + \E \left(\sum_{x \in \partial B_{R-1}} E_{0,\omega } \left[( T_{B_R} - T_{B_{R-1}})\mathbb{1}\left \lbrace X_{T_{B_{R-1}}} = x\right \rbrace \right]\right) \\
	&=E_0 T_{B_{R-1}} + \sum_{x \in \partial B_{R-1}} \E \left(E_{x,\omega } \left[ T_{B_R}\right]P_{0, \omega} \left[X_{T_{B_{R-1}}} = x \right]\right). \\
	&\le E_0 T_{B_{R-1}} + \sum_{x \in \partial B_{R-1}} E_{x } T_{B_R}.
	\end{split}
	\end{equation}
	The next step is to prove the following \textit{claim}: there exist $c$ such that, 
	\begin{equation}
	P_0\left( \tau_{1} > u \; \middle | \; D=\infty \right) \ge c P_0\left( T_{B_R} > u\right).
	\end{equation}
	In order to prove the above claim, we will follow the proof of Lemma 6.1 in \cite{FK16} doing the necessary adaptations to our case. Throughout the remainder of the proof, Figure \ref{fig:trap} will guide our arguments. 
	\begin{figure}[h]
		\centering
		\includegraphics[width=0.7\linewidth]{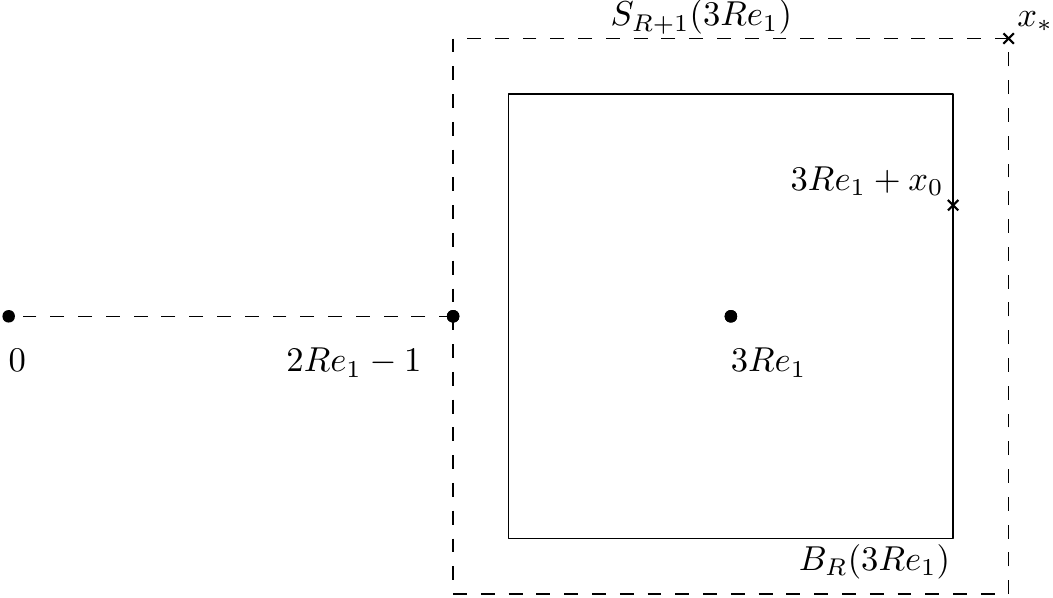}
		\caption{Regular points surrounding a potential trap $B_R(3Re_1)$.}
		\label{fig:trap}
	\end{figure}
	Before we start, we will need an additional definition. We say $x \in \Z^d$ is~$\kappa$-\textit{regular} (or simply, \textit{regular}) if $\omega(x,y) > \kappa$, for all $y \sim x$. Since we have an i.i.d elliptic environment, there exist $\kappa_0$ such that 
	\[
	\P(x \text{ is }\kappa_0\text{-regular}) >\frac{1}{2}.
	\]
	Put $\mathcal{L}$ as the line-segment $<0, e_1, \dots, (2R-1)e_1>$ and let $\mathcal{K}$ be the following event
	\begin{equation}
	\mathcal{K} := \lbrace \mathcal{L}\text{ and } S_{R+1}(3Re_1) \text{ are regular}\rbrace.
	\end{equation}
	In Figure \ref{fig:trap}, the dashed lines represent sets of regular points.
	Now, we will define several events, whose intersection will offer a lower bound the probability of $\{\tau_{1} > u, D=\infty\}$. We start by $A_1$, whose definition is self-explanatory
	\begin{equation*}
	A_1 := \lbrace X_1 = e_1, X_2 = 2e_1, \dots, X_{2R-1} = (2R-1)e_1 \rbrace;
	\end{equation*}
	Also let $A_2$ be
	\begin{equation*}
	A_2 := \lbrace X \text{ goes from }(2R-1)e_1 \text{ to }3Re_1 +x_0 \text{ using the shortest path in }S_{R+1}(3Re_1) \rbrace.
	\end{equation*}
	In more details, $A_2$ is the event in which $X$ goes from $(2R-1)e_1$ to $3Re_1 +x_0$ walking on ~$S_{R+1}(3Re_1)$ and taking the shortest path on the surface of $B_{R+1}(3Re_1)$ before jumping to $3Re_1 +x_0$, which belongs to $B_R(3Re_1)$. The next event is
	\begin{equation*}
	A_3 := \lbrace T_{B_R(3Re_1)}\circ \theta_{H_{3Re_1+x_0}} > u \rbrace.
	\end{equation*}
	I.e., after reaching the point $3Re_1+x_0 \in S_R(3R)$, the walk takes more than $u$ steps to exit $B_R(3Re_1)$.
	\begin{equation*}
	A_4 := \lbrace X \text{ goes from }X_{T_{B_R(3Re_1)}} \text{ to }2Re_1-1  \text{ using the shortest path in }S_{R+1}(3Re_1) \rbrace.
	\end{equation*}
	Once $X$ has left the smaller box $B_R(3Re_1)$ it is in the surface $S_{R+1}(3Re_1)$. Then, walking only on $S_{R+1}(3Re_1)$ and through the shortest path, the walk lands on $2Re_1 -1$.
	\begin{equation*}
	A_5 := \lbrace X_{H_{2Re_1 - 1}+1} = 2Re_1 -2, \dots,  X_{H_{2Re_1 - 1}+2R -1} = e_1 \rbrace.
	\end{equation*}
	In words, after visiting $2Re_1 -1$, the walk goes straight to $e_1$ walking on $\mathcal{L}$.. This return to $e_1$ is crucial, since it will guarantee that the regeneration does not occur before $T_{B_R(3Re_1)}$. In order to simplify the next definitions, we will write $$x_* := 3Re_1+(R+1)\sum_{i=1}^de_i.$$ The next event has a self-explanatory definition 
	\begin{equation*}
	A_6 := \lbrace X \text{ takes the shortest path from }e_1\text{ to }x_*\text{ in }\mathcal{L}\cup S_{R+1}(3Re_1) \rbrace.
	\end{equation*}
	Finally, we have our last event, 
	\begin{equation*}
	A_7 := \lbrace \text{ From }x_*, X \text{ jumps to }x_*+e_1 \text{ and never backtracks} \rbrace.
	\end{equation*}
	In other words, after reaching $x_*$, the walk takes one step at direction $e_1$ and then creates a regeneration time.
	
	Our first and crucial observation regarding the chain of events above defined is the following inclusion
	\begin{equation}
	\lbrace \tau_{1} > u, D = \infty \rbrace \supset \bigcap_{m=1}^7A_m.
	\end{equation}\label{set:tautailinclusion}
	We will conclude the proof estimating from below the probability of $A_1 \cap \dots \cap A_7$, which we do by conditioning and after several application of the Markov Property. We start from $A_7$, on $\mathcal{K}$, we have that
	\begin{equation}
	P_{0, \omega}\left[ A_7 \middle | A_6, \dots, A_1 \right] = 	P_{x_*, \omega}\left[X_1 = e_1, D \circ \theta_1 = \infty  \right] \ge  \kappa_0 P_{x_*+e_1, \omega}\left[D = \infty \right].
	\end{equation}
	Again by Markov Property, for $A_6$, on $\mathcal{K}$, we have that there exist a constant $c_6 = c(R,d,x_*)$ such that
	\begin{equation}
	\Pq{0}\left[ A_6 \middle | A_5, \dots, A_1 \right] \ge \kappa_0^{c_6},
	\end{equation}
	since the shortest path from $e_1$ to $x_*$ in $\mathcal{L} \cup S_{R+1}(3Re_1)$ is deterministic and all points of such path is regular on $\mathcal{K}$. Arguing the same way, we also conclude that there exist positive constants $c_5 = c(R)$ and $c_4 = c(R,d)$, such that
	\begin{equation}
	\Pq{0}\left[ A_5 \middle | A_4, \dots, A_1 \right] \ge \kappa_0^{c_5}; \quad \Pq{0}\left[ A_4 \middle | A_3, A_2, A_1 \right] \ge \kappa_0^{c_4}.
	\end{equation}
	Again by the Markov property,
	\begin{equation}
	\Pq{0}\left[ A_3 \middle | A_2, A_1 \right] = \Pq{3Re_1 +x_0 }\left[ T_{B_R(3Re_1)} > u\right].
	\end{equation}
	Again, arguing as for $A_6$, on $\mathcal{K}$,  there exists $c_2 = c(R,d, x_0)$,  such that
	\begin{equation}
	\Pq{0}\left[ A_2 \middle |  A_1 \right] \ge \kappa_0^{c_2}; \quad \Pq{0}\left[ A_1\right] \ge \kappa_0^{2R-1}.
	\end{equation}
	Putting all the above lower bounds together, we have that there exist positive constants $c_8 = c(R,d,x_0,x_*)$ and $c_9 = c(R,d,x_0,x_*)$  such that
	\begin{equation}
	\begin{split}
	P_0 \left(\tau_{1} > u, D= \infty\right) & \ge \kappa_0^{c_8}\E \left(\mathbb{1}_{\mathcal{K}}\Pq{3Re_1 +x_0 }\left[ T_{B_R(3Re_1)} > u\right]P_{x_*+e_1, \omega}\left[D = \infty \right]\right) \\
	& \ge \kappa_0^{c_9} P_{x_0} \left( T_{B_R} > u \right) P_0 \left( D = \infty \right),
	\end{split}
	\end{equation}
	since the random variables $\mathbb{1}_\mathcal{K}$, $\Pq{3Re_1 +x_0 }\left[ T_{B_R(3Re_1)} > u\right]$ and $P_{x_*+e_1, \omega}\left[D = \infty \right]$ are all $\P$-independent due to the independent nature of our environment. The above inequality implies that 
	$$\E\left( \tau_1 \middle | D = \infty\right) = \infty,$$
	which together with the hypothesis the walk is transient in the direction $\ell$ concludes the proof.
\end{proof}

\section*{Acknowledgments}
Alejandro Ram\'\i rez and Rodrigo Ribeiro were supported by Iniciativa
Cient\'\i fica Milenio of the Ministry of Science and Technology (Chile).
Alejandro Ram\'\i rez has been also supported by Fondecyt 1180259. The authors thank Daniel Kious for his valuable comments in the first version of this paper.


\medskip

      \begin{thebibliography}{}


\bibitem{BDR14}
 N. Berger, A. Drewitz and A.F. Ram\'\i rez. 
 \newblock Effective polynomial ballisticity conditions for random 
 walk in random environment. 
 \newblock {\em Comm. Pure Appl. Math.} 67, no. 12, 1947–1973 (2014). 

 
 \bibitem{BRS16}
 E. Bouchet, A.F. Ram\'\i rez and C. Sabot.
\newblock  Sharp ellipticity conditions for ballistic behavior of random walks in random environment. 
\newblock {\em Bernoulli} 22, no. 2, 969–994 (2016).

  
  \bibitem{CR14}
 D. Campos and A.F. Ram\'\i rez.
\newblock Ellipticity criteria for ballistic behavior of random walks
in random environment.
\newblock {\em Probab. Theory Relat. Fields} 160, no. 1-2, 189–251
(2014).


\bibitem{FK16}
 A. Fribergh and D. Kious.
\newblock Local trapping for elliptic random walks in random
environments in $\mathbb Z^d$.
\newblock {\em Probab. Theory Relat. Fields} 165, no. 3-4, 795–834 (2016).
  
\bibitem{FR20}
 R. Fukushima and A.F. Ram\'\i rez.
\newblock New high dimensional examples of ballistic random walks in
random environment.
\newblock {\em Stochastic Analysis on Large Scale Interacting Systems 2018,
RIMS K\^oky\^uroku Bessatsu}, B79, Res. Inst. Math. Sci. (RIMS), Kyoto, (2020).
  
\bibitem{GR20}
 E. Guerra and A.F. Ram\'\i rez.
 \newblock A proof of Sznitman's conjecture about ballistic RWRE. 
 \newblock {\em Comm. Pure Appl. Math.} 73, no. 10, 2087–2103 (2020).

 \bibitem{LP17}
R. Lyons and Y. Peres.
\newblock Probability on Trees and Networks.
\newblock  {\em (Cambridge Series in Statistical and Probabilistic Mathematics) }. Cambridge: Cambridge University Press. (2017).
  

 \bibitem{RS19}
 A.F. Ram\'\i rez and S. Saglietti.
\newblock New examples of ballistic RWRE in the low disorder regime. 
\newblock {\em Electron. J. Probab.} 24, Paper No. 127, 20 pp. (2019).


\bibitem{ST11} C. Sabot and L. Tournier.
  \newblock Reversed Dirichlet environment and directional
  transcience of random walks in Dirichlet environment.
  \newblock {\em Ann. Inst. H. Poincar\'e Probab. Statist.} 47:1-8 (2011).

\bibitem{ST17}
 C. Sabot and L. Tournier.
\newblock Random walks in Dirichlet environment: an overview. 
\newblock {\em Ann. Fac. Sci. Toulouse Math.} (6) 26, no. 2, 463–509 (2017).

  \bibitem{Sz01}
 A.S. Sznitman.
\newblock On a class of transient random walks in random environment.
\newblock {\em Ann. Probab.} 29, no. 2, 724–765 (2001).



\bibitem{Sz02}
 A.S. Sznitman.
\newblock An effective criterion for ballistic behavior of random walks in random environment. 
\newblock {\em Probab. Theory Relat. Fields} 122, no. 4, 509–544 (2002).

\bibitem{Sz03}
 A.S. Sznitman.
\newblock On new examples of ballistic random walks in random environment. 
\newblock {\em Ann. Probab.} 31, no. 1, 285–322 (2003).

\bibitem{Sz04}
 A.S. Sznitman.
  \newblock Topics in random walks in random environment.
  \newblock School and Conference on Probability Theory, 203–266, ICTP Lect. Notes, XVII, Abdus Salam Int. Cent. Theoret. Phys., Trieste, (2004).




\end{thebibliography}
\end{document}